\documentclass[a4paper, 12pt, hidelinks]{article}

\usepackage{package_symbols/mySymbols}
\usepackage{package_symbols/myPackages}

\usepackage[square, numbers]{natbib}

\title{Low-Rank Tensor Recovery via Theta Nuclear p-Norm}
\author{Felix Röhrich \footnote{Email: roehrich@art.rwth-aachen.de}, 
Yuhuai Zhou \footnote{Email: zhou@art.rwth-aachen.de}}
\date{}

\DeclareMathOperator{\THETABODY}{TH}
\newcommand{\thetabody}[1]{\THETABODY_{#1}}
\newcommand{\Thetabody}[1]{\widetilde{\THETABODY_{#1}}}
\newcommand{\bvee}{\bar\vee}
\newcommand{\bwedge}{\bar\wedge}
\newcommand{\tensorvec}[2]{#1^{(#2)}}

\begin{document}

\maketitle
\begin{abstract}

    We investigate the low-rank tensor recovery problem using a relaxation of the nuclear p-norm by theta bodies.
    We provide algebraic descriptions of the norms and compute their Gröbner bases. 
    Moreover, we develop geometric properties of these bodies. 
    Finally, our numerical results suggest that for
    $n\odots{\times} n$ tensors, 
    $m\geq O(n)$ measurements should be sufficient to recover low-rank tensors via theta body relaxation.
    
\end{abstract}



\section*{Introduction}

The low-rank tensor recovery problem seeks to "solve" an incomplete linear system of tensors with a low-rank solution. 
Specifically, given $A_i\in X := \Rbb^{n_1\times\cdots \times n_d}$ and $b_i\in\Rbb$ for $i=1,\cdots, m$ where $m\ll n_1\cdots n_d$,
the goal is to find a low-rank tensor $x\in X$ such that the classical linear equations hold,
\begin{equation}\label{problem: ;low rank tensor}
    \inner{A_i}{x}_F = b_i, i=1,\cdots, m,
\end{equation}
where $\inner{\cdot}{\cdot}_F$ stands for the Frobenius inner product.
When $d=2$, the problem is reduced to the low-rank matrix recovery problem.
In this case, nuclear norm minimization provides efficient algorithms to find low-rank solutions.
We refer to \cite{Chandrasekaran_Recht_Parrilo_Willsky_2012}, \cite{Foucart_Rauhut_2013} and \cite{Tropp_2015} for details.

However, computing the nuclear norm of general tensors is NP-hard (\cite{Friedland_Lim_2018},\cite{Hillar_Lim_2013}). To overcome this problem, Rauhut and Stojanac used relaxations of theta bodies in \cite{Rauhut_Stojanac_2017}, which were introduced by Gouveia et al. in \cite{Gouveia_Parrilo_Thomas_2010}.
Motivated by their result, we investigate algebraic and geometric properties of the relaxations in the more general context of nuclear $p$-norms and provide an estimate of the sufficient number of measurements required for a successful recovery.

The first step for this is to find ideals $I_p \subset \Rbb[X]$, such that the unit ball $B_p$ of the nuclear $p$-norm is the convex hull of the real algebraic variety $\Vcal_\Rbb(I_p)$, and to determine their reduced Gröbner basis; this is done in Sections \ref{section: algebraic-description} and \ref{section: Gröbner bases proofs}. This allows us to formulate the relaxations as a semidefinite program with which the recovery problem can be solved efficiently \cite{Blekherman_Parrilo_Thomas_2012}. 

In Section \ref{section: geometry} we study the geometry of theta bodies. For $p \in \{1, \infty\}$, we show that $I_p$ is real reduced and that the theta body hierarchy is finitely convergent to the nuclear $p$-norm. Afterwards, we show the following two theorems.
\begin{theorem}\label{theorem: symmetry}
    For $p\in \{1,2,\infty\}$,
    the symmetry group of $B_p$ acts invariantly on the theta bodies. In particular, the theta bodies define norms.
\end{theorem}
\begin{theorem}\label{theorem: extreme point}
    For $p\in \{1,2,\infty\}$,
    the extreme points of $B_p$ are preserved under theta body relaxations.
\end{theorem}

Finally, we address the sufficient number of measurements required for a successful recovery in Section \ref{section: sufficient number}, which is important for practical applications. In this present work, the entries of the measurements $A_i$ for $i=1,...,m$ are assumed to be generated by a standard Gaussian distribution.
Thus, for all $\epsilon\in(0,1)$,
there exists a threshold $m_0$, called \emph{the lower bound of the sufficient number of measurements}, such that if $m\geq m_0$, any tensor of rank up to $r$ can be recovered with high probability ($1-\epsilon$).
We are concerned with the question of
how $m_0$ depends on the size of the tensor $n_1,...,n_d$ and the rank $r$.
In the case of matrices ($d=2$), 
it is known that $m_0\sim O(r(n_1+n_2))$ where $r$ denotes the rank of the matrix (\cite{Tropp_2015},\cite{Chandrasekaran_Recht_Parrilo_Willsky_2012}). By using matricization, i.e. rewriting a tensor as a matrix via a vector space isomorphism, matrix nuclear norm minimization can be used to solve the recovery problem.
When $n_1=\cdots =n_d = n$, we obtain $m_0\sim O(n^{\ceil{\frac{d}{2}}})$ as a threshold in view of the above formula. For $d = 3$, however,
the numerical results of \cite{Rauhut_Stojanac_2017} suggest that recovery only requires $m\geq O(rn\log(n))$. In fact, based on our numerical results, using Gaussian width to estimate the threshold $m_0$, we conjecture:

\begin{conjecture}\label{conjecture: sufficient number}
    To recover an order-$d$ rank-$r$ tensor of size $n\times\cdots\times n$ with high-probability , 
    $m\sim O(rdn)$ measurements should be sufficient using the nuclear norm or the theta-body relaxations.
\end{conjecture}


\noindent 
\textbf{Acknowledgements}:
The authors acknowledge funding by the Deutsche Forschungsgemeinschaft (DFG, German Research Foundation) - project number 442047500 - through the Collaborative Research Center ``Sparsity and Singular Structures'' (SFB 1481).

\numberwithin{theorem}{section}
\numberwithin{equation}{section}
\section{Preliminaries}\label{section: preliminaries}
We recall and explain some basic concepts in convex geometry and semi-definite programming.
Then we introduce the construction of theta bodies and nuclear $p$-norms.

\subsection{Notation}
    We refer $m$ to the number of measurements.
    In general discussions independent of the tensor structure, 
    such as the space $\R^n$ and the polynomial ring $\R[X] = \R[x_1,...,x_n]$,
    we refer $n$ to the dimension or the number of variables.
    In cases of tensors, $d$ represents the order of the tensor;
    the bold $\nbf = (n_1,...,n_d)\in \N^d$ stands for a tuple of dimensions;
    we write $[n_i] = \set{1,...,n_i}$ and $[\nbf] = [n_1]\odots{\times}[n_d]$.
    $\R^\nbf = \R^{n_1\times\cdots\times n_d}\cong \R^{n_1}\otimes\cdots\otimes\R^{n_d}
    \cong\R^N$ where $N:= n_1\cdots n_d$.
    For varieties, we write $\Vcal$ or $\Vcal_\R$ for a real variety in the space $\R^n$ or $\R^N$
    and $\Vcal_\C$ represents the complex variety in the complex space $\C^n$ or $\C^N$.

\subsubsection*{Convex Geometry}
\begin{definition}
\quad
    \begin{itemize}
        \item A set $C\subset \R^n$ is called a \emph{convex set} if $C$ is closed under convex combinations, i.e. if $x,y\in C$, then $\forall \lambda\in[0,1],\lambda x+ (1-\lambda) y\in C $. A convex set $K$ is called a \emph{(convex) cone} if it is closed under multiplication with positive scalars, i.e. if $x\in K$ then $\forall r>0, rx\in K$.
        We assume our cones contain the origin.
        \item The \emph{dual set of a convex set} $C\in\R^n$ is defined as
        $C^*:= \set{y\in\R^n: \inner{y}{x}\geq -1, \forall x\in C}$,
        which is convex and closed. For a cone $K$, its dual set is
        $K^* = \set{y\in\R^n: \inner{y}{x}\geq 0, \forall x\in K}$,  
        which is also a cone, called the \emph{dual cone}.
        \item We call a convex set $C$ a \emph{convex body}, if $\interior(C)\not = \emptyset$.
        We also say that $C$ is full-dimensional.
        \item Let $S\subset \R^n$ be a subset. 
        The \emph{convex hull} of $S$ is the set of all convex combinations of points in $S$, i.e.
        \[
        \conv(S) = \left\{\sum_{i=1}^l \lambda_i x_i: \lambda_i \geq 0, \sum_i\lambda_i = 1, x_i\in S \right\}.
        \]
    \end{itemize}
\end{definition}

\begin{remark}
    We use the same notation for both dual set and 
    dual space $V^*$, the space of linear functions from $V$ to $\R$.
    For a vector space $V$, 
    $V^*$ will always stand for its dual space. 
    This is because the dual cone of $V$ is the trivial cone $\set{0}$.
\end{remark}
\begin{remark}
    It is also important to remark on the ambient space of the dual operator. 
    For instance, the interval $C = [0,\infty) \in \R$. 
    In $\R$, its dual is itself. 
    However when $C$ is embedded into $\R^2$ with $C = \set{(x,0):x\in[0,\infty)}$,
    its dual becomes $[0,\infty)\oplus \R$.
\end{remark}
The well-known representation theorem of convex sets will play a crucial role
in the construction of theta bodies.
\begin{theorem}\cite[Theorem 11.5]{Tyrrel_Rockerfellar_1997}\label{theorem: half-space-represetation}
    A closed convex set is the intersection of the closed half spaces containing it.
    Moreover, if $C = \cl(\conv(S))$, then the half spaces are exactly those containing $S$.
\end{theorem}

\subsubsection*{Semidefinite Program and Spectrahedron}

Semidefinite programming (SDP) serves as a powerful tool for approximating NP-hard optimization problems\cite[Section 2]{Blekherman_Parrilo_Thomas_2012}.
In particular, 
the Lasserre hierarchy (\cite{Lasserre_2001}\cite{Netzer_Plaumann_2023}\cite{Nie_2023}) generates a series 
of SDP relaxations for polynomial optimization problems defined
over semialgebraic set.

\begin{definition}[Spectrahedron and semidefinite program]\label{definition: spectrahedron}
    \quad
    \begin{itemize}
        \item A \emph{spectrahedron} $S$ is the intersection of an affine space and the cone of positive semidefinite matrices,
        i.e. for some $n, m\in\N_+$
        \begin{equation}
            S:= \set{x\in \R^n: A_0+\sum_{i=1}^n x_i A_i\succeq 0, A_i\in\R^{m\times m}\text{ symmetric}},
        \end{equation}
        where $A\succeq 0$ means \emph{positive semidefinite (p.s.d)}.
        \item A \emph{projected spectrahedron} or \emph{spectrahedral shadow} 
        is the projection of a spectrahedron onto a lower dimensional space.
        \item A \emph{semidefinite program(SDP)} is an optimization problem of a linear form over a projected spectrahedron,
        i.e.
        \[
        \min \inner{c}{x}\quad \text{subject to}\quad x\in S,
        \]
        where $c\in\R^n$ and $S\subset \R^n$ is a spectrahedral shadow.
    \end{itemize}
\end{definition}

\subsubsection*{Moment Sequence and Moment Matrix}

The concepts of moment sequences and moment matrices arise from the moment problem 
and are essential tools for analyzing positive polynomials and sums of squares.
For a comprehensive survey, please see \cite{Laurent_2009}\cite{Marshall_2008}\cite{Schmüdgen_2017}\cite{Nie_2023}.
A \emph{moment sequence} $l$ is defined as a linear functional of the polynomial ring $\R[X] = \R[x_1,...,x_n]$,
i.e. $l\in \R[X]^*$. Typically, we index $l$ by monomials. 
Namely, let $\alpha = (\alpha_1,...,\alpha_n)\in \N^n$, 
then we define $l_\alpha = l(x^\alpha) = l(\prod x_i^{\alpha_i})$ to represent the values on moments.
The \emph{moment matrix} is the infinite matrix by regarding $l$ as a bilinear form in the following manner,
\begin{equation}\label{definition: moment-matrix-bilinear-form}
    l(f,g) = l(fg), f,g\in\R[x].
\end{equation}
By default, we fix the basis of monomials for the $\R$-vector space $\R[X]$,
i.e. $\R[X] = \Span(\set{x^\alpha, \alpha\in \N^d})$.
Then the moment matrix $M(l)$ corresponding to $l$ is then defined enterwise as follow,
\[
M(l)_{\alpha, \beta} = l(x^\alpha, x^\beta) = l(x^{\alpha+\beta}).
\]
It is straightforward to see $M_{\alpha, \beta} = M_{\gamma, \iota}$ if $\alpha+\beta = \gamma+\iota$ for a moment matrix $M$.
Additionally, $M$ is symmetric.
A \emph{truncated moment sequence} refers to a linear functional on $\R[X]_{\leq t}$, 
which consists of polynomials of degree up to $t$. 
A \emph{truncated moment matrix} $M(l)$ over the space $\R[X]_{\leq t}$ is defined similarly 
by a truncated moment sequence $l\in\R[X]_{\leq 2t}^*$.

This paper focuses on the quotient algebra $\R[X]/I$ where $I\subset \R[X]$ is an ideal.
The moment sequence is then an element of $(\R[X]/I)^*$.
For $l\in (\R[X]/I)^*$, one can embed it into $\R[X]^*$ by setting $l(f) = 0$ for all $f\in I$.
The moment matrix is constructed analogously. 
The monomial basis can be derived from the theory of Gröbner basis\cite{Cox_Little_O’Shea_2015}.
Thus, we can index the moment matrix on this monomial basis. 
For truncated ones, we define the \emph{degree of polynomials in $\R[X]/I$} by
\[
\deg(f+I) = \min\set{\deg{h}: f\equiv h \mod I}.
\]
\subsection{Theta Bodies}   

Now, we are ready to discuss the theta bodies
developed in \cite{Gouveia_Parrilo_Thomas_2010}.
This approach generates a sequence of relaxations of the closed convex hull of an algebraic variety,
i.e. 
\[
    B = \overline{\conv(\Vcal_{\R}(I))}  
\]
for some ideal $I\subset\R[X]$.
Here, $\Vcal_{\R}(I)$ denotes the \emph{real algebraic variety}, which is the real zero locus of the ideal.
In this article, we omit the subscript and simply use $\Vcal$ instead of $\Vcal_\R$.
It is well known from theorem \ref{theorem: half-space-represetation} that the closed convex hull of $\Vcal(I)$
is exactly the intersection of closed half spaces containing $\Vcal(I)$, i.e.,
\[
    B = \cap_{f\text{ linear, non-negative on }\Vcal(I)}
    \{p\in\R^n: f(p)\geq 0\}.
\]
By relaxing the non-negativity condition of $f$ over $\Vcal(I)$ 
to an expression of sum of squares of degree at most $k$ modulo $I$,
we obtain a hierarchy of convex bodies known as theta bodies. 
We refer to such polynomials $f$ as \emph{$k$-sos modulo $I$}, namely
\[
    f \equiv \sum_{i=1}^l h_i^2 \mod I    
\]
where each $h_i\in\R[x_1,...,x_n]$ is a polynomial of degree at most $k$. 
We denote $V_k = (\R[X]/I)_{\leq k}:= \set{f+I\in \R[X]/I: \deg(f+I)\leq k}$
and $\Sigma_k(I):=\set{f+I: f\text{ is k-sos modulo } I}$.
The \emph{upper theta-k body of} $I$ is then defined as
\begin{equation}
        \Thetabody{k}(I) =
    \{p\in\R^n: f(p)\geq 0, \forall f\in \Sigma_k(I)\cap V_1\}.
\end{equation}

For better notation, we specify an embedding that will be frequently used in this paper.
Consider the following maps between convex sets in $\R^n$ and convex cones in $\R^{n+1}$
\begin{equation}\label{bijection-convex-cone}
    \begin{aligned}
        \text{Convex sets in }\R^n &\hookrightarrow \text{Convex cones in }\R^{n+1}\\
            C\quad\quad &\longrightarrow \quad \cone\set{(1,x)\in\R^{n+1}: x\in C}=: C^k\\
            K^c:=\set{x: (1,x)\in K} &\longleftarrow \quad K.\\
    \end{aligned}
\end{equation}
It is evident that $(C^k)^c = C$.
These operators are significant as they relate to the normalization of a cone by taking the section
at $x_0 = 1$. For example, when considering the sum of squares,
we may normalize its constant term to be 1 (if it exists). 
In this manner, we can rewrite the upper theta bodies in the following way
\begin{equation}\label{definition: upper theta body}
    \Thetabody{k}(I) =
    \{p\in\R^n: f(p)\geq 0, \forall f\in \Sigma_k(I)\cap V_1\} = ((\Sigma_k(I)\cap V_1)^*)^c.
\end{equation}
Here the dual operator is taken in the dual space $V_1^*$.
And the $^c$ operator means that the functional sends constants to themselves,
i.e. $l(1) = 1$.
The terminology differs from that used in the original work(\cite{Gouveia_Parrilo_Thomas_2010})
and we will explain this later.
Similarly, let $P(I):= \set{f\in\R[X]: f|_{\Vcal(I)}\geq 0}$,
then we can rephrase $B = ((P(I)\cap V_1)^*)^c$.
Now let $\Sigma_k(I)^*$ be the dual of $\Sigma_k(I)\subset V_{2k}$ in $V_{2k}^*$.
Then we restate \cite[Theorem 2.8]{Gouveia_Parrilo_Thomas_2010} as follows

\begin{theorem}\label{Theorem: theta body projective spectrahedra}
    If $\Sigma_k(I)$ is closed, then
    \begin{equation}\label{definition: thetabody}
        \Thetabody{k}(I) = \thetabody{k}(I):= \cl \pi_x((\Sigma_k(I)^*)^c).
    \end{equation}  
    Here the dual operator is taken in the dual space $V_{2k}^*$.
    We refer to this $\thetabody{k}(I)$ as the \emph{theta-k body of} $I$.
\end{theorem}
In general, we have
\begin{equation}\label{inequality: thetabody-underline}
    \thetabody{k}(I)\subset\Thetabody{k}(I).
\end{equation}
This inclusion is apparent because $\Sigma_k(I)\cap V_1\subset \Sigma_k(I)$ and
the dual operator is inclusive-reversing. 
This theorem is crucial as it allows us to formulate $\thetabody{k}(I)$
as a semidefinite program. 
We assume that $\thetabody{k}(I)$ is symmetric;
more specifically, it defines a norm via the following;
for $x\in\R^n$,
\[
    \norm{x}_{\thetabody{k}(I)} := \inf\{r\geq 0: x\in r\thetabody{k}(I)\}.   
\]

Computing this norm is a semidefinite program. 
First, we argue that $\Sigma_k(I)^*$ is a spectrahedron.
It follows that $\Thetabody{k}(I)$ is a projected spectrahedron.
Indeed, let $S$ be the space of symmetric matrices (truncated moment matrix) on $V_k$ and 
let $S_+$ denote those that are positive semidefinite in $S$.
Then $p\in \Sigma_k(I)^*\subset V_{2k}^*\subset S$ is equivalently saying $p$ is p.s.d as a quadratic form, i.e.
\begin{equation}\label{Remark: spectrahedra of theta bodies}
    \Sigma_k(I)^* = S_+\cap V_{2k}^*.
\end{equation}
Thus, we conclude that $\Sigma_k(I)^*$ is indeed a spectrahedron. 
Therefore, theta norm minimization is a semidefinite program. 
Similarly, the upper theta bodies $\Thetabody{k}(I)$ can be formulated in terms of a semidefinite program.
However, it is more complicated to determine the affine restriction. 
In Section \ref{section: dual body and k-sos}, we will provide the semidefinite representation of $\Sigma_k(I)$.
Moreover, since $\thetabody{k}(I)$ is a weaker relaxation than
$\Thetabody{k}(I)$ due to the inclusion relation mentioned above,
our implementation will focus on $\thetabody{k}(I)$.

In general, for a compact variety, 
Schmuedgen's Positivestellensatz ensures the convergence of theta bodies \cite[Thm 7.32]{Gouveia_Thomas_2012}.

\begin{theorem}\label{Thm: convergence-thetabody}
    If $\Vcal(I)\subset\R^n$ is compact, then $\Thetabody{k}(I)$ converges to $\conv(\Vcal(I))$ for $k \to \infty$. Hence $\thetabody{k}(I)\mathop{\rightarrow}\limits^{k\rightarrow\infty} \conv(\Vcal(I))$ as well.
\end{theorem}

\begin{definition}[Theta-exact]
    These concepts describe the finite convergence of the theta body hierarchy.
    \begin{enumerate}
        \item We say $I$ is $\thetabody{k}$-exact or theta-k-exact,
    if $\thetabody{k}(I) = \overline{\conv(\Vcal(I))}$.
    We say $I$ is theta-exact, if there exists $k>0$ such that
    $I$ is theta-k-exact.
        \item Similarly, we say $I$ is $\Thetabody{k}$-exact or upper-theta-k-exact, if $\Thetabody{k}(I) =\overline{\conv(\Vcal(I))}$.
    And we say $I$ is upper-theta-exact, if there exists $k>0$ such that
    $I$ is upper-theta-k-exact.
    \end{enumerate}
    Clearly, upper-$\thetabody{k}$-exactness implies $\thetabody{k}$-exactness.
\end{definition}

In \cite{Gouveia_Thomas_2012} and \cite{Gouveia_Parrilo_Thomas_2010},
the authors describe $\Thetabody{1}(I)$ using convex quadrics.
We say $f\in\R[x_1,...,x_n]$ is a quadric,
if it has the form $f(x) = x^T A x + Bx + c$ with $A\succeq 0$.
\begin{theorem}\label{Theorem: theta1-convex-quadrics}
    Let $I\subset\R^n$ be any ideal, then
    \[
    \Thetabody{1}(I) = \cap_{f\in I\text{, convex quadrics}}\{p\in\R^n: f(p)\leq 0\}.    
    \]
\end{theorem}
This is valuable because it provides a concrete description of the (upper) theta bodies.
For those ideals $I$ that may not satisfy closedness of $\Sigma_k(I)$,
understanding $\Thetabody{k}(I)$ may still contribute
since $\thetabody{k}(I)\subset \Thetabody{k}(I)$.

\subsection{Nuclear $p$-Norms}\label{section: nuclear-$p$}
Recall a \emph{rank-1 tensor} $X\in\R^{n_1}\otimes\cdots\otimes \R^{n_d}$ has the form
\[
    X = x^{(1)}\otimes\cdots\otimes x^{(d)}
\]
where $x^{(i)}\in\R^{n_i}, i=1,...,d$.
A tensor can be identified as a $d$-array; 
namely, $\R^{n_1}\otimes\cdots\otimes \R^{n_d} \cong \R^{n_1\times....\times n_d}$ via
\[
    (x^{(1)}\otimes\cdots\otimes x^{(d)})[i_1,...,i_d] =  \prod_{k=1}^d x^{(k)}_{i_k}. 
\]
The \emph{nuclear norm} of a tensor $X\in\R^{n_1\times....\times n_d}$ is defined as follows,
\[
\norm{X}_* = \min\left\{\sum_{i=1}^r\norm{X_i}_2: X = \sum_{i=1}^r X_i, X_i\in\R^{n_1\times....\times n_d},
\rank(X_i) = 1 \right\}.    
\]
where $\norm{\cdot}_2$ denotes the $l_2$-norm by regarding the tensor as a $\R^{n_1\cdots n_d}$ vector.
It turns out that the unit ball of the nuclear norm is
the convex hull of the set of rank-1 tensors with unit $l_2$-norm.
That is 
\begin{align*}
    B_2 &= \conv{\{x^{(1)}\otimes...\otimes x^{(d)}\in\Rbb^{n_1\times...\times n_d}: \norm{x^{(i)}}_2 = 1\}}\\
    &= \conv{\{X\in\Rbb^{n_1\times...\times n_d}: \norm{X}_2 = 1, \rank(X) = 1\}}.
\end{align*}
The $l_p$-norm of a tensor is defined analogously as the $l_2$-norm.
The nuclear $p$-norm generalizes the nuclear norm by extending $l_2$-norm to $l_p$ norms.
First of all, we demonstrate that the following two generalizations coincide, 
i.e. $B_p^1 = B_p^2$ defined as follows
\begin{align}\label{Bp}
&B_p^1 = \conv{\{x^{(1)}\otimes...\otimes x^{(d)}\in\Rbb^{n_1\times...\times n_d}: \norm{x^{(i)}}_p = 1\}}\\
&B_p^2 = \conv{\{X\in\Rbb^{n_1\times...\times n_d}: \norm{X}_p = 1, \rank(X) = 1\}}.
\end{align}

\begin{lemma}
    For any natural number $p\geq 1$ and $p=\infty$, $B_p^1 = B_p^2$.
\end{lemma}
\begin{proof}
    The key observation is that for any $p\geq 1$,
    \begin{align*}
        \norm{x^{(1)}\otimes...\otimes x^{(d)}}_p^p &= \sum_{i_1,...,i_d}|x^{(1)}_{i_1}...x^{(d)}_{i_d}|^p\\
        &= \sum_{i_1}|x^{(1)}_{i_1}|^p\sum_{i_2}...\sum_{i_d}|x^{(d)}_{i_d}|^p\\
        &= \norm{x^{(1)}}_p^p...\norm{x^{(d)}}_p^p
    \end{align*}
    this induces that the extreme points of $B_p^1$ and $B_p^2$ coincide and so do their convex hull.
    A similar argument works for $p=\infty$.
\end{proof}
We identify both bodies as $B_p$.
The corresponding norm is defined by its gauge function; that is,
\begin{equation}\label{definition: gauge}
        \gamma_C(x):= \inf\set{r>0: x\in rC}.
\end{equation}
We refer them to the \emph{nuclear $p$-norms}\cite{Friedland_Lim_2018}.

\section{Algebraic Description of Nuclear $p$-Norm}
\label{section: Gröbner}
In this section, 
we provide algebraic descriptions of nuclear $p$-norms,
namely we find ideals $I_p$ such that $B_p = \conv(\Vcal_\R(I_p))$ 
for some reasonable values of $p$.
Indeed, we will prove that when $p=1,\infty$ and all positive even numbers,
such ideals exist.

Recall the construction of theta bodies,
\[
\thetabody{k}(I) = \cl \pi_x((\Sigma_k(I)^*)^c).
\]
A functional $p\in\Sigma_k(I)^*\subset V_{2k}^*:=(\R[X]/I)_{\leq 2k}^*$ is also called a truncated moment sequence on $V_{2k}$.
Moreover, by considering it as a bilinear form on $V_k:= (\R[X]/I)_{\leq k}$
in the following manner
\[
p(f+I, g+I) = p(fg+I),
\]
we relate the truncated moment sequence to a (symmetric) truncated moment matrix
once we fix a suitable basis for $V_k$.
The non-negativity of $p$ on $\Sigma_k(I)$
is translated to the positive semidefinite-ness of this moment matrix on $V_k$.
A suitable basis should have the property that 
if $\Bcal_k$ is a basis for $V_k$, 
then $\Bcal_k\cdot \Bcal_k$ spans $V_{2k}$.
The Gröbner basis theory provides us with such a suitable basis by monomials.
Therefore, after establishing the ideals $I_p$, 
we compute their Gröbner basis $\Gcal_p$.
Section \ref{section: intro-Gröbner} is devoted to providing a brief overview on the topic of Gröbner bases; for details, we refer to \cite{Cox_Little_O’Shea_2015}.

For the remainder of the section, we fix the following notation.
Let $d\in \Nbb$ and $\nbf = (n_1, \dots, n_d) \in \Nbb^d$. We write $[\nbf]$ for $[n_1]\times \dots \times [n_d]$, where $[n_i]  = \{1, \dots, n_i\}$. For elements $a,b \in [\nbf]$ we write $a < b$, if $a$ is smaller than $b$ with respect to the lexicographic order. The tensor product $X = \R^{n_1}\odots{\otimes}\R^{n_d}$ comes with a canonical basis $\{ \textbf{e}_a \mid a \in [\nbf] \}$ given by the tensor product of the standard bases of $\Rbb^{n_1}, \dots, \Rbb^{n_d}$. Hence, we can write $x\in X$ uniquely as $x = \sum_{a\in [\nbf]} x_a e_a$ and obtain an isomorphism $X \cong \Rbb^{n_1\cdots n_d}$.
We consider $\Rbb[X]$, the polynomial ring over $\Rbb$ with variables $x_a$ where $a$ ranges over $[\nbf]$. For $\alpha \in \N_0^{[\nbf]}$, we write
\[
    x^\alpha = \prod_{a\in[\nbf]} x_a^{\alpha_a} = x_{1\dots1}^{\alpha_{1\dots1}} \cdots x_{n_1\dots n_d}^{\alpha_{n_1\ldots n_d}}
\]
Note that, a priori, the polynomial ring $\Rbb[X]$ no longer sees the multilinear structure of $X$ and only the linear structure of $\Rbb^{n_1\cdots n_d}$, but it is still possible to describe certain properties, such as being rank 1, by polynomial equations.

\subsection{Algebraic Description}
\label{section: algebraic-description}

Recall the definition of the nuclear $p$-norm through its unit ball $B_p$,
\[
    B_p = 
    \conv{\{X\in\Rbb^{n_1\times...\times n_d}: \norm{X}_p = 1, \rank(X) = 1\}}.
\]

We will begin with an algebraic description of the condition $\rank(X) = 1$. For this, we will need the following notation: Let $a, b\in\N^d$, we write
\[
    a\wedge b = (\min(a_i, b_i))_{i=1,\dots,d},
    \quad
    a\vee b = (\max(a_i, b_i))_{i=1,\dots,d}.
\]


\begin{lemma}
    \label{lem-rk-1-eq}
    A nonzero tensor $x\in \R^\nbf$ is rank 1, if and only if $x_a x_b - x_{a\wedge b}x_{a\vee b} = 0$ for all $a, b \in [\nbf]$. 
    In other words, the set of rank-1 tensors is the variety of the ideal
    \[I_0 = \gen{x_a x_b - x_{a\wedge b}x_{a\vee b}, \forall a,b\in[\nbf]}.\]
\end{lemma}
\begin{proof}
    First, assume that $x$ is rank 1, i.e. $x = x^{(1)} \otimes \dots \otimes x^{(d)}$. Then by definition
    \[
        x_{a\wedge b} = x^{(1)}_{\min\{a_1, b_1\}} \cdots x^{(d)}_{\min\{a_d, b_d\}} \text{ and } x_{a\vee b} = x^{(1)}_{\max\{a_1, b_1\}} \cdots x^{(d)}_{\max\{a_d, b_d\}},
    \]
    thus $x_a x_b - x_{a\wedge b} x_{a\vee b} = 0$.

    To show that $x$ is rank 1 if $x_a x_b - x_{a\wedge b}x_{a\vee b} = 0$, it is sufficient to find $x^{(1)} \in \Rbb^{\nbf_1}, \dots, x^{(d)} \in \Rbb^{\nbf_d}$, such that
    \[
        x_a = (x^{(1)} \otimes \dots \otimes x^{(d)})_a
    \]
    for all $a\in [n]$. Since $x \neq 0$, there exists some $s\in [n]$ with $x_s \neq 0$. Put $x^{(1)} = (x_{1 s_2\dots s_d}, \dots, x_{n_1 s_2\dots s_d})$ and $x^{(i)} = (x_{s_1\dots 1\dots s_d}/x_s, \dots, x_{s_1\dots n_i\dots s_d}/x_s)$ for $i > 1$. Now
    \[
        (x^{(1)} \otimes \dots \otimes x^{(d)})_a = x^{(1)}_{a_1} \cdots x^{(d)}_{a_d} = (x_s)^{-d+1} x_{a_1 s_2\dots s_d} \cdots x_{s_1\dots s_{d-1} a_d}.
    \]
    By assumption, we have the identity
    \[
        x_{a_1 s_2\dots s_d} x_{s_1 a_2\dots s_d} = x_{s_1 s_2\dots s_d} x_{a_1 a_2\dots s_d} = x_s  x_{a_1 a_2\dots s_d}.
    \]
    Thus, the above expression reduces to $(x_s)^{-d+2} x_{a_1 a_2\dots s_d} \dots x_{s_1\dots s_{d-1} a_d}$. Iterating the argument, we obtain the desired identity.
\end{proof}

Now we proceed to find the defining polynomials of the extreme points for the condition $\norm{X}_p = 1$.

For $p = 1$, the extremal points are the standard basis vectors $e_a$. It follows that every coordinate is in $\{ -1, 0, 1 \}$, giving the equations $x_a^3 - x_a = 0$ for all $a\in [\nbf]$. Additionally, only one coordinate should not vanish, so $x_a x_b = 0$ for $a \neq b$, and to ensure that exactly one entry is nonzero we require $\sum_{a\in \nbf} x_a^2 - 1$. Altogether we find
\begin{equation}
    I_1 = \gen{ x_a^3 - x_a, x_a x_b, \sum\nolimits_{a\in [\nbf]} x_a^2 - 1}
\end{equation}

If $2 \le p < \infty$ is an even number, the defining polynomial is clear since $\norm{X}_p = 1$ is equivalent to $\sum_{a\in[\nbf]}x_a^{p}-1 = 0$. We obtain
\begin{equation}\label{definition: Ip}
    I_p = \gen{\sum_{a\in[\nbf]}x_a^{p}-1, x_a x_b - x_{a\wedge b} x_{a\vee b}, \forall a,b\in[\nbf]}.
\end{equation}
For odd numbers, $I_p$ is not the desired ideal since 
$\Vcal_\R(I_p)$ in this case is unbounded.

For $p = \infty$, every coordinate of the extreme points is either $-1$ or $+1$ giving us
\begin{equation}
    I_\infty = \gen{x_a^2-1, x_a x_b - x_{a\wedge b} x_{a\vee b}, \forall a,b\in[\nbf]}.
\end{equation}

\subsection{An Introduction to Gröbner Bases}\label{section: intro-Gröbner}
We first need to fix a total order on the monomials. Here we use the \emph{graded reverese lexicographic order}, \emph{grevlex} for short (sometimes also \emph{degrevlex}), where $x^\alpha > x^\beta$ if the first nonzero entry from the right of $\alpha - \beta$ is negative. In our case, this means $x_{1\dots 11} > x_{1\dots 12} > \dots > x_{21\dots 1} > \dots > x_{\nbf_1\dots \nbf_d}$. With respect to such an ordering, for a nonzero polynomial $f = \sum a_\alpha x^{\alpha}$ we have its
\begin{itemize}
    \item multidegree $\multideg(f) = \max\{ \alpha \mid a_\alpha \neq 0\}$,
    \item leading coefficient $\LC(f) = a_{\multideg(f)}$,
    \item leading monomial $\LM(f) = x^{\multideg(f)}$ and
    \item leading term $\LT(f) = \LC(f)\LM(f)$.
\end{itemize}

Given an ordered set $\Gcal = \{ g_1, \dots, g_n \} \subset \Rbb[X]$, every $f\in \Rbb$ can be written as
\[
    f = c_1 g_1 + \ldots + c_n g_n + r
\]
with $c_1, \dots, c_n, r\in \Rbb[X]$ and $\multideg(r) < \multideg(g_i)$ using a division algorithm. We say $f$ reduces to $r$ modulo $\Gcal$, $f \rightarrow_\Gcal r$.

In general, the remainder $r$ depends on the ordering of $g_1, \dots, g_n$, in particular if $\Gcal$ is an arbitrary generating set of some ideal $I \subset \Rbb[X]$, it cannot be used in conjunction with the division algorithm to determine ideal membership.

The set $\Gcal$ is called Gröbner basis of $I$ (with respect to the monomial ordering) if $(\LM(\Gcal)) = (\LM(I))$. In this case, it has the important property that $r$ is unique, in the sense that it no longer depends on the order of $\Gcal$. Hence, for a polynomial $f\in \Rbb[X]$ we have $f\in I$ if and only if $r$ is zero, i.e. $f \rightarrow_\Gcal 0$.

To determine whether $\Gcal$ is a Gröbner basis can be done via $S$-polynomials. The $S$-polynomial of $f_1, f_2\in \Rbb[X]$ is defined as
\[
    S(f_1, f_2) = \frac{\lcm(\LT(f_1), \LT(f_2))}{\LT(f_1)} f_1 - \frac{\lcm(\LT(f_1), \LT(f_2))}{\LT(f_2)} f_2
\] 

\begin{theorem}[Buchberger's Criterion]
    \label{thm-buchberger-criterion}
    A subset $\Gcal = \{ g_1, \dots, g_n \}$ of an ideal $I\subset \Rbb[X]$ is a Gröbner basis of $I$ if and only if $\Gcal$ generates $I$ and
    \[
        S(g_i, g_j) \rightarrow_\Gcal 0
    \]
    for all $g_i, g_j \in \Gcal$.
\end{theorem}

$\Gcal$ is called \emph{reduced} if for all $g\in \Gcal$ we have $\LC(g) = 1$ and no monomial of $g$ is contained in $\LT(\Gcal\setminus \{g\})$.

\begin{theorem}
    Let $\Gcal = \set{g_1,\dots, g_n}$ be a Gröbner basis for an ideal $I\subset\R[X]$,
    then the monomials not divisible by any of $\set{\LT(g_i), g_i\in\Gcal}$ span the space $\R[X]/I$. In particular, those monomials of degree up to $k$ span the space $(\R[X]/I)_{\leq k}$.
\end{theorem}

\subsection{Gröbner basis}
\label{section: Gröbner bases proofs}
The goal of this section is to proove the following theorem case by case.
\begin{theorem}
    \label{theorem:groebner-basis}
    We have the following reduced Gröbner basis
    \begin{itemize}
        \item $\Gcal_0 := \{ x_a x_b - x_{a\wedge b}x_{a\vee b} \mid a, b \in [\nbf], a < b \text{ and } a_i > b_i \text{ for some } i \}$ for $I_0$,
        \item $\Gcal_{1} := \{ x_a x_b \mid a, b \in [\nbf], a < b \} \cup \{ \sum_{a\in [\nbf]} x_a^2 - 1 \} \cup \{ x_a^3 - x_a \mid a \in [\nbf], a > (1, \dots, 1) \}$ for $I_1$,
        \item $\Gcal_{p} := \{ \sum_{a\in [\nbf]} x_a^p - 1 \} \cup \Gcal_0$ for $I_p$ with $2 \le p < \infty$, and
        \item $\Gcal_\infty := \{ x_a^2 - 1 \mid a \in [\nbf] \} \cup \{ x_a x_b - x_{a\bwedge b} x_{a\bvee b} \mid a, b \in [\nbf], a < b \text{ and } a_i > b_i \text{ or } a_i = b_i < \nbf_i \text{ for some } i \}$ for $I_{\infty}$,
    \end{itemize}
\end{theorem}

where we use the following notations for $a,b\in [\nbf]$
\[
    (a\wedge b)_i = \min\{a_i, b_i\}, \quad (a\vee b)_i = \max\{a_i, b_i\}
\]
and
\[
    (a\bwedge b)_i = \begin{cases}
        \min\{a_i, b_i\} & \text{if } a_i \neq b_i, \\
        n_i & \text{else.}
    \end{cases},
    \quad
    (a\bvee b)_i = \begin{cases}
        \max\{a_i, b_i\} & \text{if } a_i \neq b_i, \\
        n_i & \text{else.}
    \end{cases}
\]

The case for $p = 2$ was already proven in \cite{Rauhut_Stojanac_2017}. The case $2 \leq p < \infty$ is based on their work, except for the fact that they relied on matricization to prove their results, while we work directly with tensors. The main benefit is an easier description of $\Gcal_{p}$.

For the other cases, we used the computer algebra system OSCAR \cite{OSCAR} to compute examples for the Gröbner bases of the ideals $I_p$.

\begin{proposition}
    \label{pro-rk-1-poly-div}
    Let  $x_{a^1} \cdots x_{a^k} \in \Rbb[X]$. Then $x_{a^1} \cdots x_{a^k}$ reduces to $ x_{b^1}\cdots x_{b^k}$ modulo $\Gcal_0$,
    where $b^j_i = \min  (\sqcup_{s=1}^k \{ a^s_i \} \setminus \sqcup_{s=1}^{j-1} \{ b^s_i \})$.
\end{proposition}
\begin{proof}
    Assume that $x_{a^1} \cdots x_{a^k} \in \Rbb[X]$ is divisible by $x_{a^1}x_{a^2} - x_{a^1\wedge a^2}x_{a^1\vee a^2}$. Polynomial division by $x_{a^1}x_{a^2} - x_{a^1\wedge a^2}x_{a^1\vee a^2}$ can now be purely expressed by an operation on the tuples $(a^1_i, \dots, a^k_i)$ for $1\le i \le d$. In particular, the operations for different indices are independent of each other. Thus, we can reduce to the case $d = 1$, i.e. $X = \Rbb^n$. Now, it is clear that the operation is simply sorting $a^1, \dots, a^k$ ascending.
\end{proof}

\begin{corollary}
    \label{cor-I0-Gröbner-basis}
    $\Gcal_0$ is the reduced Gröbner basis of $I_0$.
\end{corollary}
\begin{proof}
    Note that $x_a x_b - x_{a\wedge b} x_{a\vee b}$ and $x_c x_d - x_{c\wedge d} x_{c\vee d}$ are coprime if $\{a, b\}$ and $\{c, d\}$ are disjoint. We will prove the case $b = d$, the other cases are similar. Compute
    \[
        S(x_a x_b - x_{a\wedge b}x_{a\vee b}, x_c x_b - x_{c\wedge b}x_{c\vee b}) = x_a x_{c\wedge b} x_{c\vee b} - x_c x_{a\wedge b} x_{a\vee b}
    \]
    and compare the index sets for both monomials
    \[
        \{ a_i, (c\wedge b)_i, (c\vee b)_i \} = \{ a_i, b_i, c_i \} = \{ c_i, (a\wedge b)_i, (a\vee b)_i \}
    \]
    By \ref{pro-rk-1-poly-div} we have $S(g_{a,b}, g_{a,c}) \rightarrow_{\Gcal_0} 0$, hence $\Gcal_0$ is a Gröbner basis of $I_0$. The reducedness is clear.
\end{proof}

\begin{corollary}
    $\Gcal_{p}$ is the reduced Gröbner basis of $I_{p}$ for $2 \le p < \infty$.
\end{corollary}
\begin{proof}
    The polynomials $\sum_{a\in [\nbf]} x_a^p - 1$ and $x_a x_b - x_{a\wedge b} x_{a\vee b}$ are always coprime, since their respective leading terms are $x_{1\dots1}^p$ and $x_a x_b$ with $a \neq b$, so the statement follows from \ref{cor-I0-Gröbner-basis}.
\end{proof}

\begin{proposition}
    $\Gcal_1$ is the reduced Gröbner basis of $I_1$.
\end{proposition}
\begin{proof}
    We begin with the computation of the $S$-polynomials. All polynomials in the first set are monomials; thus, their S-polynomials are necessarily 0. Furthermore, all leading monomials of the second and third set are coprime and consequently, their S-polynomials are 0. It remains to check
    \[
        S(x_a x_b, x_a^3 - x_a) = x_a x_b \rightarrow_{\Gcal_1} 0,
        \quad
        S(x_a x_b, x_b^3 - x_b) = x_a x_b \rightarrow_{\Gcal_1} 0,
    \]
    and
    \[
        S(x_{1\dots 1} x_b, \sum\nolimits_{a\in [\nbf]} x_a^2 - 1) = \left(\sum\nolimits_{a > (1, \dots, 1), a \neq b} x_a^2 x_b\right)  + x_b^3 - x_b \rightarrow_{\Gcal_1} 0.
    \]
    One easily sees that $\Gcal_1$ is a generating set of $I_1$ and thus $\Gcal_1$ is a Gröbner basis of $I_1$; the reducedness is clear.
\end{proof}

\begin{lemma}
    For $a, b\in [\nbf]$, we have $x_a x_b - x_{a\bwedge b} x_{a\bvee b} \in I_{\infty}$.
\end{lemma}
\begin{proof}
    The term $x_a x_b - x_{a\bwedge b} x_{a\bvee b}$ is either $0$ or in $\Gcal_\infty$. The first case is clear, so let us consider the second.
    
	We assume that $a_i = b_i < \nbf_i$ for some $i$, as otherwise $a\bwedge b = a\wedge b$ and $a\bvee b$ = $a\vee b$, hence the claim follows from the definition of $I_\infty$. Without restriction we can assume $x_a > x_b$, hence $a_j < b_j$ for some $j < i$. Together, this gives $(a\bwedge b)_i = \nbf_i > b_i = a_i\vee b_i$ and $(a\bwedge b)_j = a_j < b_j = a_j\vee b_j$, so $x_{a\bwedge b} x_{a\vee b} - x_{a\wedge b}x_{a\bvee b} \in I_{\infty}$, where we used the fact that $a\wedge b = (a\bwedge b) \wedge (a\vee b)$ and $a\bvee b = (a\bwedge b) \vee (a\vee b)$. Now,
    \begin{align*}
        x_a x_b - x_{a\bwedge b} x_{a\bvee b} &=
         x_a x_b - x_{a\wedge b} x_{a\vee b} + x_{a\wedge b} x_{a\bwedge b}(x_{a\bwedge b} x_{a\vee b} - x_{a\wedge b} x_{a\bvee b}) \\
        & + x_{a\bwedge b} x_{a\bvee b}(x_{a\wedge b}^2 - 1) - x_{a\wedge b} x_{a\vee b}(x_{a\bwedge b}^2 - 1)
    \end{align*}
\end{proof}

\begin{lemma}
    \label{lem-Ginfty-generates-Iinfty}
    $\Gcal_{\infty}$ generates $I_{\infty}$.
\end{lemma}
\begin{proof}
	It is sufficient to show that we can write $x_a x_b - x_{a\wedge b} x_{a\vee b}$ in terms of elements of $\Gcal_{\infty}$, but this is simply
    \[
        (x_a x_b - x_{a\bwedge b} x_{a\bvee b}) - (x_{a\wedge b} x_{a\vee b} - x_{(a\wedge b)\bwedge (a\vee b)} x_{(a\wedge b)\bvee (a\vee b)}).
    \]
\end{proof}

\begin{proposition}
    \label{pro-Ginfty-poly-div}
    Let $x_{a^1}\dots x_{a^k} \in \Rbb[X]$. For $j\in [d]$ and $i \in [n_j]$ put $s_{i,j} := \#\{ i = a^l_j \mid 1 \le l \le k\}$ and write $s_{i,j} = l_{i,j} + \delta_{i,j}$ with $l_{i,j}$ even and $\delta_{i,j} \in \{0, 1\}$.
    Let $l = k - \min_j \sum_i l_{i,j}$ and $S_j = \{ i \mid \delta_{i,j} = 1 \}$. For $1 \le r \le l$ we put $b^r_j = \min (S_j \setminus \cup_{s=1}^{r-1} \{ b^s_j \}) \cup \{ n_j \}$. Then $x_{a^1}\dots x_{a^k}$ reduces to $x_{b^1} \dots x_{b^l}$ modulo $\Gcal_\infty$.
\end{proposition}
\begin{proof}
    Similarly to \ref{pro-rk-1-poly-div}, one can show that $x_{a^1}\dots x_{a^k}$ modulo
    \[
        R = \{ x_a x_b - x_{a\bwedge b} x_{a\bvee b} \mid a, b \in [\nbf], a < b \text{ and } a_i > b_i \text{ or } a_i = b_i < n_i \text{ for some } i \}
    \]
    reduces to $x_{b^1}\dots x_{b^k}$ where $b^r_j = \min ((S_j \setminus \cup_{s=1}^{r-1} \{ b^s_j \}) \cup \{ n_j \})$ for $1 \le r \le k$.

    Now assume that $x_{a^1}\cdots x_{a^k}$ is already reduced modulo $R$. Clearly, it is only divisible by $x_\nbf^2 - 1$ and exactly $\frac12(\min_j \sum_i l_{i,j})$ times.
\end{proof}

\begin{corollary}
    $\Gcal_{\infty}$ is the reduced Gröbner basis of $I_{\infty}$.
\end{corollary}
\begin{proof}
    \ref{lem-Ginfty-generates-Iinfty} asserts that $\Gcal_\infty$ generates $I_\infty$ and by \ref{pro-Ginfty-poly-div} computation of the S-polynomials gives
    \[
        S(x_a^2 - 1, x_a x_b - x_{a\bwedge b} x_{a\bvee b})
        = x_a x_{a\bwedge b} x_{a\bvee b} - x_b \rightarrow_{\Gcal_\infty} 0,
    \]
    since we have (in the notation of \ref{pro-Ginfty-poly-div}) $S_j = \{ b_j \}$ and $l = 2$ for $x_a x_{a\bwedge b} x_{a\bvee b}$,
    and similarly
    \[
        S(x_a x_b - x_{a\bwedge b} x_{a\bvee b}, x_a x_c - x_{a\bwedge c} x_{a\bvee c})
        = x_b x_{a\bwedge c} x_{a\bvee c} - x_c x_{a\bwedge b} x_{a\bvee b} \rightarrow_{\Gcal_\infty} 0
    \]    
    Hence, $\Gcal_{\infty}$ is a Gröbner basis for $I_{\infty}$. One easily checks that $G_{\infty}$ is also reduced.
\end{proof}

\subsection{The moment matrix for theta-1 body}\label{section: moment matrix for theta-1}

The Gröbner basis enables us to 
characterize the moment matrix for theta-k bodies of $I_p$.
In particular, 
we give the characterization for theta-1 bodies of $I_2$ and $I_\infty$.
Fix the tensor space as $\R^{n_1\times\cdots \times n_d}$. 
Denote $[\nbf] = [n_1]\times\cdots\times[n_d]$ and $N:= n_1\cdots n_d$.
In both cases, 
$\Bcal_1 = \set{1, x_a, a\in[\nbf]}$ spans $(\R[X]/I_p)_{\leq 1}$.
Abbreviate $I = I_p$. 
The moment matrix for theta-1 body represents a quadratic form on $(\R[X]/I)_{\leq 1}$.
We form it as
$M = (M_{a,b})\in\R^{(N+1)\times (N+1)}$ for $a, b\in\Bcal_1$.
We use index $0$ for constants, that is $x_0 = 1\in\Bcal_1$.
For $p = 2$, the moment matrix is then symmetric and
\begin{equation*}
    \Mcal_2 = \set{ M\in\R^{(N+1)\times (N+1)}:
    M_{0,0} = \sum_{a\in [\nbf]}M_{a,a},
    M_{a, b} = M_{a\wedge b, a\vee b}, \forall a\not = b\in [\nbf].}
\end{equation*}
For $p = \infty$, the moment matrix is also symmetric and
\begin{equation*}
    \Mcal_\infty = \set{ M\in\R^{(N+1)\times (N+1)}:
    M_{0,0} = M_{a,a}, 
    M_{a, b} = M_{\wedge(a, b), \vee(a,b)}, \forall a\not = b\in [\nbf]}.
\end{equation*}
Abbreviate both as $\Mcal$. 
Computing the corresponding theta-1 norm of $x\in\R^{\nbf}$ can be formulated as follows
\begin{equation*}
    \min M_{0,0}\quad
    \text{s.t. } M\in\Mcal, M\succeq 0,
    M_{0,a} = x_a, \forall a\in[\nbf]
\end{equation*}
This motivates the term \emph{reduced pair}. 
\begin{definition}\label{def: reduced-pair}
    For any pair $(a,b)\subset [\nbf]$ that $a\not = b$,
we say it is \emph{reduced} if $(a,b) = (a\wedge b, a\vee b)$ up to change of positions.
And we call $(a\wedge b, a\vee b)$ the \emph{reduced form} of $(a,b)$.
Similarly, an analog is defined for $p=\infty$.
\end{definition}

\section{Geometry of Theta bodies}\label{section: geometry}
This section investigates the geometric properties of the bodies $\thetabody{k}(I_p)$, 
specifically focusing on (real) algebraic and convex geometry.
We pay particular attention to cases where $p=1,2,\infty$.
First, we address the real reduced-ness of the ideals in Section \ref{section: reducedness}.
Recall theorem \ref{Theorem: theta body projective spectrahedra},
\begin{theorem}
    If $\Sigma_k(I)$ is closed, then
    \begin{equation}
        \Thetabody{k}(I) = \thetabody{k}(I):= \cl \pi_x((\Sigma_k(I)^*)^c).
    \end{equation}  
\end{theorem}
$\Sigma_k(I)$ is closed if $I$ is radical or real radical, see \cite[Corollary 2.9]{Gouveia_Parrilo_Thomas_2010} and \cite[Proposition 2.6]{Powers_Scheiderer_2001}.
We will prove in Section \ref{section: reducedness} that $I_1, I_\infty$
are real reduced.
However, this remains unverified for other $I_p$.

In Section \ref{section: Theta exactness}, we study the convergence of theta bodies. 
It turns out that $\thetabody{k}(I_1)$ coincide with $l_1$-norm; 
hence it is theta-1 exact due to its real reducedness.
The theta body hierarchy of $I_\infty$ is also finitely convergent,
with respect to the tensor size.
However, it is not clear whether $I_2$ is theta-exact. 
Although in the matrix case, $I_2$ is in fact theta-1 exact,
this may not hold in general. 
Otherwise, we find an algorithm to compute the nuclear norm efficiently. 

In Section \ref{section: symmetry}-\ref{section: extreme points},
we explore the symmetries and extreme points of the theta bodies. 
Namely, theorem \ref{theorem: symmetry} and theorem \ref{theorem: extreme point}
will be discussed here. 
These results imply the basic intuition that the theta bodies also define norms, 
and they are not far from the original norm.

Finally, in Section \ref{section: dual body and k-sos},
we provide the semidefinite program characterization of the dual bodies of theta bodies,
along with $\Sigma_k(I)$.

We use both $\Vcal_\R(I)$ and $\Vcal(I)$ as the real variety of $I$ in $\R^n$ and
$\Vcal_\C(I)$ the complex variety of $I$ in $\C^n$.
\subsection{Reducedness of $I_p$}\label{section: reducedness}

A polynomial $f$ is equivalent to a $k$-sos modulo $I$ 
implies that $f$ is the same as a $k$-sos polynomial on the variety $\Vcal(I)$. 
This suggests considering the vanishing ideal of the variety $\Vcal(I)$
--- the ideal containing all polynomials vanishing on the variety.
Below we list definitions from real algebraic geometry.
\begin{definition}\label{Definition: radical}
    Let $I\subset \R[X]$ be an ideal.
    \begin{itemize}
        \item \emph{Radical of $I$},
        $
        \sqrt{I} := \set{f\in \R[X]: \exists n\in\N, f^n\in I}.    
        $
        \item \emph{Real radical of $I$},
        $
        \Rradical{I}:= \set{f\in \R[X]: \exists n,m\in\N, g_i\in\R[X], f^{2n} + \sum_{i}^m g_i^2\in I}.
        $
        \item We say $I$ is \emph{radical} or \emph{reduced} if $I = \sqrt{I}$, and similarly
        \emph{real radical} or \emph{real reduced} if $I = \Rradical{I}$.
    \end{itemize}
\end{definition}
Real Nullstellensatz states that the real radical of $I\subset\R[X]$ is indeed the
vanishing ideal on $\Vcal(I)$.
We show that in the following both $I_1, I_\infty$ are real radical.

\begin{lemma}[Seidenberg's Lemma]\cite[Proposition 3.7.15]{Kreuzer_Robbiano_2000}
    Fix a field $K$. Let $I\subset K[x_1,...,x_n]$ be an ideal defining a finite variety.
    Suppose for any $i\in \set{1,...,n}$, there exists a non-zero polynomial 
    $f_i\in I\cap K[x_i]$ such that $f_i, f_i'$ are prime to each other in $\R[x_i]$.
    Then $I$ is radical.
\end{lemma}

\begin{corollary}
    Let $p = 1,\infty$, recall that $\Gcal_p$ is the Gröbner basis of $I_p$.
    Define $I_\C(\Gcal_p) = \gen{\Gcal_p}$ as the ideal generated by $\Gcal_p$ in $\C[X]$.
    Then $I_\C(\Gcal_p)$ is radical.
\end{corollary}

\begin{proof}
    For $p = \infty (p = 1)$, it is obvious that $\Vcal_\C(I_p) = \Vcal_(I_p)$ is finite.
    For each index $a$, consider $f_a(x_a) = x_a^2-1 \in I_p\cap \C[x_a]$ 
    ($f_a(x_a) = x_a^3-x_a$).
    Then $f_a, f_a'$ are prime to each other in $\C[x_a]$. 
    From Seidenberg's Lemma, it follows $I_\C(\Gcal_p)$ is radical.
\end{proof}

\begin{proposition}
    Let $p = 1, \infty$, the ideal $I_p\subset\R[X]$ is real radical.
\end{proposition}

\begin{proof}
    Let $\Gcal = \Gcal_p$, $I_\C = I_\C(\Gcal_p), I = I_p$.
    Observe that $\Vcal_\C(I_\C) = \Vcal(I)$, i.e.
    all complex solutions of $\Gcal_p$ are real.
    Hilbert's Nullstellensatz implies for any $f\in\R[X]$ vanishing on $\Vcal(I) = \Vcal_\C(I)$,
    $f\in I_\C$. Since the Gröbner basis is independent of the field
    and $\Gcal\subset \R[X]$, 
    the division algorithm results in an expression of $f$ that
    $f=\sum_{g(x)\in \Gcal}c_g(x)g(x)$
    with $c_g(x)$ are all real and it follows that $f\in I\subset \R[X]$.
    Therefore $I$ is in fact the vanishing ideal of its real variety. 
    By Real Nullstellensatz, $I$ is real radical.
\end{proof}

Although we are unable to establish the (real) reducedness of $I_2$,
the following lemma shows that $\thetabody{1}(I_2)\supset \thetabody{1}(\Rradical{I_2})$.
With further observations regarding symmetries and convex geometry of $\thetabody{1}(I_2)$,
it remains meaningful to analyze theta bodies associated with $I_2$. 
Fortunately, the numerical results based on $I_2$ suggest that we do not miss much information
in this relaxation.

\begin{lemma}
    Let $I\subset J\subset\R[X]$ be ideals, then $\thetabody{k}(I)\supset \thetabody{k}(J)$,
    $\Thetabody{k}(I)\supset \Thetabody{k}(J)$.
\end{lemma}

\begin{proof}
    By construction, if a polynomial $f$ is $k$-sos modulo $I$,
    then it is $k$-sos modulo $J$. It follows $\thetabody{k}(I)\supset \thetabody{k}(J)$.
    A similar argument applies to upper theta bodies.
\end{proof}

\subsection{Theta-Exactness}\label{section: Theta exactness}

We show the theta-exactness of $I_p$ for $p = 1,\infty$.

\begin{proposition}
    $I_1$ is theta-1 exact.
\end{proposition}

\begin{proof}
    $l_1$ unit ball is also known as the cross polytope.
    By \cite[Example 4.6]{Gouveia_Parrilo_Thomas_2010},
    its vanishing ideal is theta-1 exact.
    This proposition then follows from the fact $I_1$ is real radical.
\end{proof}

\begin{proposition}\label{Proposition: real radical I infty}
    Suppose that the tensor space is of size $\R^{n_1\times\cdots \times n_d}$, let $n=\max\set{n_i, i=1,...,d}$,
    then $I_\infty$ is theta-$n$ exact.
\end{proposition}

\begin{proof}
    It suffices to show that any polynomial of degree greater than $n$
    is equivalent to a polynomial of lower degree modulo $I:=I_\infty$. 
    Indeed, if this is true,
    then every sum of squares is equivalent to an $n$-sos polynomial modulo $I$.
    
    Regard any polynomial as an element in the vector space $\R[X]/I$, 
    which is spanned by monomials not divisible by the initial monomials in $\Gcal := \Gcal_\infty$.
    A monomial is not divisible by initial monomials in $\Gcal$ if and only if
    it satisfies the following two conditions;
    first, it contains no powers, 
    which means it has the form $x_{a^{(1)}}x_{a^{(2)}}\dots x_{a^{(k)}}$
    where $a^{(i)}\not=a^{(j)}$ for $i\not = j$; 
    second,
    every pair $(a,b)$ of $J = \set{\tensorvec{a}{1},...,\tensorvec{a}{k}}$ 
    is reduced in the sense that
    $(a, b) = (a\bwedge b, a\bvee b)$ up to change of positions 
    (from Section \ref{section: moment matrix for theta-1}).
    Then the highest degree of such a monomial $x_{a^{(1)}}x_{a^{(2)}}\dots x_{a^{(k)}}$
    equals the largest cardinality of $J= \set{\tensorvec{a}{1},...,\tensorvec{a}{k}}$
    such that every pair is reduced. 
    
    We now show that this $J$ has at most cardinality $n = \max\set{n_i, i=1,...,d}$.
    By construction of the notation $\bwedge$ and $\bvee$,
    a pair of indices $(a,b)$ is reduced 
    if for any $i\in \set{1,...,d}$, 
    either $a_i< b_i$ or $a_i= b_i = n_i$.
    To achieve maximality, $\gamma := (n_1,..., n_d)\in J$. 
    Otherwise, by adding this index, one always extends the set $J$. 
    Then for any other $b\in J$,
    there exists $i\in\set{1,...,d}$ such that $b_i < n_i$.
    Assume $i=1$,
    consider $a\in J$ other than $\gamma, b$ such that $a_1<b_1<n_1$.
    Indeed, if $a_1 = b_1$, then they must both equal $n_1$.
    Continue in this way, the maximal cardinality is $n_1$. 
    If $n_1 = n$, then
    $J$ has maximal length $n$. 
    Then the basis of $\R[X]/I$ has maximal degree $n$,
    which means any polynomial equals a polynomial of degree at most $n$ modulo $I$. 
    The proof is complete.
\end{proof}

\subsection{Symmetries of Theta bodies}\label{section: symmetry}

We answer here that the symmetries of the variety continue to act invariantly on the (upper) theta bodies or their corresponding unit spheres,
provided that the ideal is indeed the vanishing ideal.
Although the real reducedness of $I_2$ is still unclear,
we can still prove the same statement with additional insights.

We first discuss generally about the group actions on varieties, ideals and the coordinate rings.

\begin{lemma}\label{action in Rx/I}
    Let a group $G\subset \GLgroup(n)$ act on $\Vcal(I)\subset \R^n$ and $I$ is real radical.
    Then $G$ also acts naturally on $\R[X]$ and $\R[X]/I$ in the following manner;
    for any $f\in\R[X]$
    \[
    g\cdot f(x) = f(g^{-1}x),\quad g\cdot (f+I) = (g\cdot f)+I
    \]
    In particular, $G$ is invariant in $I$.
\end{lemma}

\begin{proof}
    Assume $f\in I$, then $g\cdot f$ vanishes on $\Vcal(I)$ as $g^{-1}$ acts on $\Vcal(I)$.
    Since $I$ is real radical, from the Real Nullstellensatz it follows that $g\cdot f\in I$. 
    This implies that the action $g\cdot (f+I)$ is well defined.
\end{proof}

Recall the definition of theta bodies given by equation (\ref{definition: thetabody}), 
\[\thetabody{k}(I):= \cl \pi_x((\Sigma_k(I)^*)^c).\]
We have a natural action of $G$ on $\thetabody{k}(I)$;
namely, the dual action. 
For $l\in V_{2k}^*, g\in G, f\in V_{2k} = (\R[X]/I)_{\leq 2k}$
\begin{equation}\label{action in V*}
    g\cdot l(f) := l(g^{-1}\cdot f) = l\circ f \circ g.
\end{equation}
This action is well-defined on $V_{2k}^*$ due to lemma \ref{action in Rx/I}. 
For any point $u\in \pi_x((\Sigma_k(I)^*)^c)$,
there exists $l_u\in\Sigma_k(I)^*\subset V_{2k}^*$ such that $l_u(1) = 1$,
$\pi_x(l_u) = u$. 
The action of $g$ on $u$ is defined as
\[
g\cdot u:= \pi_x(g\cdot l_u).
\]
This is well-defined as for any $l\in V_{2k}^*$ such that $l(1) = 1, \pi_x(l) = u$,
by regarding $x_i$ as the coordinate polynomial, we have
\[
\pi_{x_i}(g\cdot l) = g\cdot l (x_i) = l(g^{-1}\cdot x_i) = l\circ x_i(gx) = (gu)_i.
\]
Recall that $gu$ is the action of $g$ on $\R^n$.
For the closure operator, we can use the action on sequences.
Suppose $u = \lim u_t$, then $g\cdot u = \lim g\cdot u_t$.
It is well-defined since the linear map is continuous.
In particular, in our examples of nuclear $p$-norms, 
we can eliminate the closure operator
by the following lemma.
\begin{lemma}\label{lemma: theta body without cl}
    For $I = I_2, I_\infty$, $(\Sigma_k(I)^*)^c\subset V_{2k}^*$ is compact (in Euclidean topology).
    In particular, $\cl \pi_x((\Sigma_k(I)^*)^c) = \pi_x((\Sigma_k(I)^*)^c).$
\end{lemma}
\begin{proof}
    It is well-known in convex geometry\cite{Tyrrel_Rockerfellar_1997} that the dual operator maps to a closed convex set.
    Then its section $(\Sigma_k(I)^*)^c$ is closed. It suffices now to show it is bounded.
    Indeed, regard $l\in V_{2k}^*$ as a symmetric matrix on $V_{k}$, say $Q_l$. 
    Taking the section means $l(1) = Q_l(1,1) = 1$. 
    In $I_2$, the relation $\sum_{a}x^2_a = 1$ indicates that
    the trace of $Q_l$ equals 2. 
    As $Q_l$ is p.s.d, its diagonal entries are non-negative,
    hence bounded. 
    For $I = I_\infty$,  the relation $x_a^2 = 1$ implies that 
    the diagonal entries of $Q_l$ are all $1$. 
    Then from the non-negativity of principal 2-by-2 minors, 
    it follows that every entry of $Q_l$ should be bounded. 
    This enables us to remove the closure operator 
    since the continuity of projection preserves compactness. 
\end{proof}
We note here that this action in fact
coincides with the action in $\R^n$.
To summarize, we present the following lemma.
\begin{lemma}
    The action of $G$ in $\pi_x((V_{2k}^*)^c)$ defined above coincides with 
    the action in $\R^n$. In particular, it is invariant in $\thetabody{k}(I)$.
\end{lemma}
\begin{proof}
    The discussion above leaves only invariant action on $\thetabody{k}(I)$,
    which is equivalent to showing $\Sigma_k(I)^*$ is invariant. 
    It is clear since every $f\in \Sigma_k(I)$, $g\cdot f\in\Sigma_k(I)$.
\end{proof}

A similar action can also be interpreted on the upper theta bodies 
$\Thetabody{k}(I) =((\Sigma_k(I)\cap V_1)^*)^c$,
which we omit here.

Now, we can explain that if $\conv(\Vcal(I))\subset \R^n$ defines 
a norm---meaning that it is full dimensional and symmetric---then so are the theta bodies.
Indeed, since the variety is symmetric, $g = -\Id\in \GLgroup(n)$ acts on the variety, where $\Id$
is the identity map.
Given that $I$ is invariant under $g$, then $g$ also acts on the theta bodies.
It follows that the theta bodies define norms. 
Moreover, the norm is also invariant under this group action.

\subsubsection*{Symmetries of theta bodies of nuclear $p$-norms}

Here, we consider the symmetries of theta bodies for $I = I_2, I_\infty$.
We have already established in proposition \ref{Proposition: real radical I infty} that
$I_\infty$ is real radical. 
Therefore based on previous discussions, 
the symmetries of $\conv(\Vcal(I_\infty))$ are also symmetries of its theta bodies.
Recall that
\[
B_p = \conv{\{x^{(1)}\otimes...\otimes x^{(d)}\in\Rbb^{n_1\times...\times n_d}: \norm{x^{(i)}}_p = 1\}}
\]
Let $C_n$ denote the cubes in $\R^n$.
Then for the tensor space $\R^{n_1\odots{\times} n_d}$, 
$B_\infty$ is the convex hull of $C_{n_1}\otimes...\otimes C_{n_d}$.
Denote $H_n$ the symmetry group of the hypercube $C_n$,
which consists of changes of positions or signs.
Then $G_\infty := H_{n_1}\times\cdots\times H_{n_d}$ acts naturally on $B_\infty$ via 
for $x = \sum_{i=1}^r x^{(i_1)}\otimes\cdots\otimes x^{(i_d)}, x^{(i_j)}\in \R^{n_j}, h_i\in H_{n_i}$,
\begin{equation}
    h_1\otimes\cdots\otimes h_d\cdot x = \sum_{i=1}^r h_1\cdot\tensorvec{x}{i_1}\otimes\cdots\otimes h_d\cdot x^{(i_d)}.
\end{equation}

For $I = I_2$, we first discuss the symmetries of $B_2$.
In fact, $B_2$ is the convex hull of $S_{n_1}\otimes\cdots\otimes S_{n_d}$,
where $S_n$ is the Euclidean unit sphere of $\R^n$.
Similarly, $S_n$ has the orthogonal group $\Ogroup(n)$ as its symmetry group.
The product of the groups acts on $B_2$ in the same manner.
Although we do not have a result about real reducedness, 
by lemma \ref{action in Rx/I}, 
it suffices to show $G_2 := \Ogroup(n_1)\otimes\cdots\otimes \Ogroup(n_d)$ is invariant in $I$.
First, recall that $I_0$ is the ideal generated by rank-1 binomials in $I_2$.

\begin{lemma}
     $I_0$ is real radical in $\R[X]$.
\end{lemma}

\begin{proof}
    We know that $\Gcal_0$ is the set of binomials in $\Gcal_2$ and $\Gcal_0$ is a Gröbner basis of $I_0$ by corollary \ref{cor-I0-Gröbner-basis}.
    Consider $f^2\in I_0$, we want to show $f\in I_0$. 
    Suppose $f\not\in I_0$, 
    $f\rightarrow_{\Gcal_0} r$
    where $r$ is not divisible by any $x_a x_b \in\LT(\Gcal_0)$.
    In particular, $x_ax_b \not| \LT(r)$.
    However, $f^2\in I$ implies that $r^2\in I$.
    In other words, there exists $x_ax_b\in \LT(\Gcal_0)$ 
    that divides $\LT(r^2)$. 
    Since $a\not = b$, $x_ax_b|\LT(r)$,
    which is a contradiction.
    
    Now consider any sum of squares $\sum_{i=1}^l f_i^2\in I$.
    We want to show any $f_i\in I$.
    Let $r_i$ be the remainder of $f_i$ through the division algorithm by $\Gcal$.
    Equivalently, $\sum_i r_i^2\in I$. 
    W.L.O.G, we may assume $\LT(r_1)>\LT(r_2)>...>\LT(r_l)$ in the monomial order.
    Similarly, there exists a monomial $x_ax_b\in \LT(\Gcal_0)$ 
    such that $x_a x_b|\LT(\sum_i r_i^2) = \LT(r_1^2)$; hence $x_ax_b|\LT(r_1)$.
    If $r_1\not= 0$, it is a contradiction.
    Then $r_1 = 0$ and therefore $r_i = 0$, which means $f_i\in I$.
    The conclusion $I_0$ is real radical follows 
    from \cite[Proposition 12.5.1]{Marshall_2008}.
\end{proof}

\begin{proposition}
    The ideal $I_2$ is invariant under $G_2 = \Ogroup(n_1)\otimes\cdots\otimes \Ogroup(n_d)$.
\end{proposition}

\begin{proof}
    Define $f_2(x):= \sum_{a} x_a^2 - 1$.
    We know that $I_2 = I_0 + \gen{f_2}$.
    For any $f\in I_2$, we have the following decomposition
    \[
        f(x) = c(x)f_2(x) + h(x)
    \]
    where $c(x)\in\R[X], h(x)\in I$. Then for any $g\in G_2$,
    \[
        g\cdot f(x) = c(g^{-1}x)f_2(g^{-1}x) + h(g^{-1}x).    
    \]
    Note that $g$ preserves the $l_2$-norm, therefore $f_2(g^{-1}x) = f_2(x)$.
    As $I$ is real radical and $g$ preserves the rank of $x$, $g\cdot h$ will vanish on $\Vcal(I)$,
    hence $g\cdot h\in I$. Now it follows $g\cdot f\in I_2$.
\end{proof}

\begin{remark}
    This proposition also holds for all positive even numbers.
\end{remark}

Finally, we conclude that $G_p$ acts on the theta bodies of $B_p$,
for $p\geq 2$ even or $p=\infty$.
Moreover, upon verifying that $\gen{\sum_a x_a^p-1}$ is real radical,
we may propose the following conjecture.
\begin{conjecture}
    The ideal $I_p$ is real radical for all even numbers $p\geq 2$.
\end{conjecture}

\subsection{Extreme points of theta bodies}\label{section: extreme points}

Recall that
we say $x\in B$ is an extreme point of the closed convex set $B$,
if for any $y,z\in B$ such that $x = \lambda y+ (1-\lambda) z = x$ for some $\lambda \in(0,1)$
then $x = y = z$. 
The extreme points of $B_p = \conv(\Vcal(I_p))$ are exactly the variety points.
We will show that the variety $\Vcal(I_p)$ for $p=2,\infty$ remain extreme points of (upper) theta bodies. 
By the inclusion $\Thetabody{k}(I)\supset \thetabody{k}(I)\supset \cl\conv{\Vcal(I)}$,
it suffices to work only with upper theta bodies.
Recall that we can use convex quadrics to characterize $\Thetabody{1}(I)$(theorem \ref{Theorem: theta1-convex-quadrics}),
\begin{equation}\label{equation: convex-quadrics}
    \Thetabody{1}(I) = \cap_{f\in I\text{, convex quadrics}}\{p\in\R^n: f(p)\leq 0\}. 
\end{equation}
The convex quadrics has the form $f(x) = x^TAx + Bx + c$ where $A$ is p.s.d.

\begin{theorem}
    If the set of convex quadrics in $I$, 
    say $\Fcal = \set{f(x) = x^TA_f x + B_fx + C_f\in I, A_f \text{ p.s.d }}$ so that
    \[
    \bigcap_{f\in \Fcal}\ker(A_f) = \set{0},    
    \]
    then the extreme points of $\cl\conv(\Vcal(I))$ are also extreme points for $\Thetabody{k}(I)$ and $\thetabody{k}(I)$ for any $k\geq 1$.
\end{theorem}

\begin{proof}
    Let $x\in \Vcal(I)$ be an extreme point of $\cl\conv(\Vcal(I))$, 
    then for any convex quadrics in $I$, $f(x) = 0$.
    Suppose $x = \lambda y + (1-\lambda)z$ for some $\lambda\in (0,1), y,z\in \Thetabody{1}(I)$.
    By equation \ref{equation: convex-quadrics}, for any convex quadrics $f\in I$,
    $f(y), f(z)\leq 0$. 
    The convexity of $f$ implies that $f(x)\leq \lambda f(y) + (1-\lambda) f(z)$.
    Then necessarily $f(y) = f(z) = f(x) = 0$. 
    If $y\not = z$, it follows that $f$ should be linear on the direction $y-z$,
    which is equivalent to $y-z\in \ker(A_f)$ for any $f\in\Fcal$. 
    Now if the intersection of the kernels is trivial, 
    it follows $y=z=x$ and therefore $x$ is extreme of $\Thetabody{1}(I)$.
    Then for any $k\geq 1$, $\Thetabody{1}(I)\supset\Thetabody{k}(I)\supset \cl\conv(\Vcal(I))$
    and $\Thetabody{1}(I)\supset \thetabody{1}(I)\supset \thetabody{k}(I)\supset \cl\conv(\Vcal(I))$,
    $x$ is also extreme for any (upper) theta-k body.
\end{proof}

\begin{corollary}
    For $I = I_2, I_\infty$, the real zeros $\Vcal(I)$ are extreme points of $\thetabody{k}(I)$ for $k\geq 1$. 
\end{corollary}

\begin{proof}
    Consider convex quadrics of form $\sum_a x_a^2 + C$.
\end{proof}
In what follows, we study the characterization of the other extreme points for theta-1 body
as a projected spectrahedron (lemma \ref{lemma: theta body without cl}).
Specifically, we have
\begin{equation}
    \thetabody{k}(I) = \pi_x((\Sigma_k(I)^*)^c).    
\end{equation}
The authors of \cite{Ramana_Goldman_1995} investigated the representations of faces of a spectrahedron. 
Recall that a spectrahedron is an intersection of an affine space and the cone of p.s.d matrices
(Definition \ref{definition: spectrahedron}).
Fix symmetric matrix $A_0,...,A_n\in\R^{m\times m}$, 
a spectrahedron has the form
\[
    S:= \set{l\in \R^n: A_0+\sum_{i=1}^n l_i A_i\succeq 0}.  
\]
We denote the quadratic form associated with the functional $l$ as $Q_l:= A_0 + \sum_{i=1}^n l_i A_i$. 
Since $\Sigma_k(I)^*$ is a cone, 
we set $A_0 = 0$. We deduce the lemma from \cite{Ramana_Goldman_1995},

\begin{lemma}
    $Q_l$ spans an extreme ray of a spectrahedron $S$ if and only if it has a maximal kernel,
    i.e. if $p\in\R^n$ s.t. $\ker(Q_p)\supset \ker(Q_l)$ then $Q_p = r Q_l$ for some $r\in\R$.
\end{lemma}

Except for the trivial case,
rank-1 quadratic forms will have maximal kernel.
In the following, we show that these rank-1 forms correspond exactly to those points in the variety
(\cite[Section 4.6]{Blekherman_Parrilo_Thomas_2012}).
Indeed, any $x\in \Vcal(I)$ defines a linear form $l_x\in (\R[X]/I)^*$ through evaluation,
i.e. $l_x(f) = f(x)$. They are the rank-1 quadratics on any $V_k$.

\begin{corollary}
    $Q_l$ spans an extreme ray of $\Sigma_k(I)^*$, then it is either rank 1 or 
    the real zeros of its kernel have no intersection with $\Vcal_\R(I)$.
    Furthermore, the rank-1 quadratics $Q_l$ on $V_k:= (\R[X]/I)_{\leq k}$ are exactly the quardratic forms induced by 
    points on $\Vcal(I)$.
\end{corollary}

\begin{proof}
    It is clear that $l\in (V_{2k})^*$ then necessarily for any $f\in I$ satisfying degree bound, 
    $l(f) = 0$. Then $l_x\in (V_{2k})^*$ implies $x\in\Vcal(I)$.
    The statements follow from the proof of \cite[Corollary 4.40, Lemma 4.41]{Blekherman_Parrilo_Thomas_2012}. 
\end{proof}

\subsection{Dual body of theta bodies and testifying k-sos}\label{section: dual body and k-sos}

Recall that if $\Sigma_k(I)$ is closed,
the theta bodies equal the upper theta bodies,
\[
    \Thetabody{k}(I) = ((\Sigma_k(I)\cap V_1)^*)^c.
\]
\begin{lemma}\label{lemma: section commutative with dual}
    The operators $^c, ^k$ commute with the dual operator $^*$.
    In particular, 
    \begin{equation}\label{dual-theta-body}
    \Thetabody{k}(I)^* = (\Sigma_k(I)\cap V_1)^c.
    \end{equation}
\end{lemma}

\begin{proof}
    Here we prove only for $^c$, the other one is similar.
    Let $K\in\R^{n+1}$ be a cone, we need to show $(K^c)^* = (K^*)^c$.
    In fact, $y\in (K^c)^*$ if and only if for all $x\in K^c$
    $
    \inner{y}{x}\geq -1.
    $
    Equivalently, $\inner{(1,y)}{(1,x)} \geq 0$ for all $x\in K^c$.
    That is indeed $(1, y)\in K^*$. The statement follows.
\end{proof}
It follows that for any vector $u\in\R^N$, 
$u$ is in the dual upper theta-k body if and only if
$1+\inner{u}{x}\in\Sigma_k(I)\subset\R[X]/I$. 
Thus, the membership problem of the dual body
is equivalent to the membership problem of $k$-sos polynomials.
This is indeed a generalization of \cite[Section 3.1.4]{Blekherman_Parrilo_Thomas_2012}.
Let $p(x)+I\in \R[X]/I$, we can express it in terms of a monomial basis of $\R[X]/I$ relative to a Gröbner basis.
For this purpose, the Gröbner basis is supposed to be based on a graded order of monomials.
Let $\R[X] = \R[x_1,...,x_n]$ and $\Bcal\subset \N^n$ such that $x^\Bcal:= \set{x^\alpha, \alpha\in\Bcal}$
span the space $\R[X]/I$, i.e. 
each monomial is not divisible by any initial monomial in the Gröbner basis of $I$.
Then,
\[p(x) = \sum_{\alpha\in \Bcal}p_\alpha x^\alpha.\]
We can test whether $p(x)+I\in \Sigma_k(I)$ using a semidefinite program. 
Let $\Bcal_k := \set{\alpha\in\Bcal: |\alpha|\leq k}$ or equivalently $x^{\Bcal_k}$ spans $V_k = (\R[X]/I)_{\leq k}$.
Then $p(x)+I\in \Sigma_k(I)$ if and only if there exists a p.s.d matrix $Q$ such that
\begin{equation}
    p(x) = [x^{\Bcal_k}]^T Q x^{\Bcal_k}, \quad Q\succeq 0
\end{equation}
where $Q$ is compatible with the Gröbner basis of $I$. 
See the following example for illustration.

\begin{example}
    Consider the 2-by-2 matrix in the form
    $
    \begin{pmatrix}
        x_1 & x_2\\
        x_3 & x_4
    \end{pmatrix}.    
    $
    For $I = I_\infty$, its Gröbner basis is 
    $\Gcal = \Gcal_\infty = \set{x_1^2-1, x_2^2-1, x_3^2-1, x_4^2-1, x_1x_2-x_3x_4, x_1x_3-x_2x_4, x_2x_3-x_1x_4}$.
    Then $x^\Bcal = \set{1, x_1, x_2, x_3, x_4, x_1x_4, x_2x_4, x_3x_4}$.
    Consider $p(x) = 1 + \inner{u}{x}$. Write Q in the form
    \[
    Q = \begin{pmatrix}
        Q_{00} & Q_{01} & Q_{02} & Q_{03} & Q_{04}\\
        Q_{10} & Q_{11} & Q_{12} & Q_{13} & Q_{14}\\
        Q_{20} & Q_{21} & Q_{22} & Q_{23} & Q_{24}\\
        Q_{30} & Q_{31} & Q_{32} & Q_{33} & Q_{34}\\
        Q_{40} & Q_{41} & Q_{42} & Q_{43} & Q_{44}
    \end{pmatrix},
    \quad Q\succeq 0.
    \]
    To be consistent with the Gröbner basis, 
    \begin{align*}
        1 &= Q_{00}+Q_{11}+Q_{22}+Q_{33}+Q_{44},\\
        u_i &= 2Q_{0i},\quad i = 1,2,3,4,\\
        Q_{12}+Q_{34} &= Q_{13}+Q_{24} = Q_{14}+Q_{23} = 0.
    \end{align*}
    The first row relates to the constant of $p(x)$.
    The second row contributes to the coefficients of $x_i$.
    The third row encodes coefficients of $x_3x_4, x_2x_4, x_1x_4$ respectively,
    which in our case are all zero.
    This is indeed a projected spectrahedron, and finding decompositions is equivalent to finding points in this spectrahedral shadow.
\end{example}

\section{Lower bound for the sufficient number of measurements}\label{section: sufficient number}
We provide our numerical estimation on the lower bound of sufficient number of measurements
for recovering rank-1 tensors 
(\cite{Tropp_2015}, \cite{Chandrasekaran_Recht_Parrilo_Willsky_2012}). 
That is, our goal is to identify a threshold $m_0$ in relation to $n_1,...,n_d$ and $\epsilon\in(0,1)$, 
such that for any $m\geq m_0$, 
\[\Prob{\text{any tensor of rank up to r can be recovered}}\geq 1-\epsilon.\]
Particularly, we are concerned about the dependence on $n_1,...,n_d$.

In Section \ref{section: intro-Gaussian-width}, 
we first recall some cones related to a convex set and introduce the concept of Gaussian width.
The Gaussian width is an efficient way to measure the size of convex sets in high-dimensional spaces.
Please see \cite{Chandrasekaran_Recht_Parrilo_Willsky_2012}\cite{Tropp_2015}\cite{Vershynin_2018} for more details.
Furthermore, we explain how the sufficient number of measurements relates to this concept.
In Section \ref{section: computation-Gaussian-width},
we present our compute process, 
which is motivated by the computation in \cite[Section 4.4]{Tropp_2015}.
However, we cannot finish the theoretical computation with an ideal result
but stop at a point where numerical experiments can be applied.
Therefore, we provide our numerical results in Section \ref{section: numerical-results}.

\subsection{Gaussian width of cones}\label{section: intro-Gaussian-width}
We first introduce some notions about cones.
\begin{definition}\label{definition: tangent/normal-cone}
    Let $C\in\R^n$ be a closed convex set (body) and any $x\in \R^n$.
    \begin{itemize}
        \item The \emph{tangent cone} or \emph{descent cone}  of $C$ at $x$.
        \[
            D(C, x):= \set{v\in\R^n: \exists t>0, \forall r\in (0,t), x+rv\in C}.
        \]
        \item The \emph{normal cone} of $C$ at $x$.
        \[
            N(C, x):= -D(C, x)^*  = \set{v\in\R^n: \inner{v}{y-x}\leq 0, \forall y\in C}.    
        \]
    \end{itemize}
\end{definition}
\begin{remark}\label{remark: descent-cone}
    \quad
    \begin{enumerate}
        \item These cones usually make sense only when $x$ is on the relative boundary of $C$. 
        If $x\not\in C$, then $D(C,x) = \set{0}$. 
        When $x$ is in the relative interior of $C$, $T(C,x)$ is the affine hull of $C$.
        In particular, if $C$ is full dimensional and $x\in\interior(C)$, $D(C,x) = \R^n$.
        Now suppose $C$ is a convex body with $0\in\interior(C)$.
        We allow scalars for $K = D, N$ with the following notation. 
        For any $x\in \R^n\setminus\set{0}$,
        \[
        K(C, [x]) := K(\gamma_C(x)C, x) = K(C, \frac{x}{\gamma_C(x)}).  
        \]
        With this, we actually have $\forall r>0, x\not = 0$, $K(C,[x]) = K(C,[rx])$.
        \item The well-known bi-duality theorem states that
        \begin{equation}\label{equation: T = -N*}
            D(C,[x]) = -N(C,[x])^*
        \end{equation}
    \end{enumerate}
\end{remark}
We give the next lemma without proof, which characterizes the normal cones.
\begin{lemma}\label{lemma: normal-cone-equiv}
    Let $C$ be a convex set containing the origin, $x$ is on the relative boundary of $C$, then $v\in N(C, x)$ if and only if $\exists r > 0$ so that
    \[
    -rv\in C^*, \inner{rv}{x} = -1.
    \]
\end{lemma}

\begin{definition}
    Let $K\in\R^n$ be a convex cone. 
    The Gaussian width of $K$ is defined as
    \begin{equation}
        w(K):= \Exp{\sup_{u\in K\cap S^{n-1}}\inner{u}{g}}{g}
    \end{equation}
    where $g\in \R^n$ is a standard Gaussian.
\end{definition}

Now we connect this concept with the recovery problem in a general setting.
Let $C$ be the unit ball of some norm $\norm{\cdot}_C$,
Our recovery algorithm solves the following optimization program
\begin{equation}\label{recovery-problem-norm}
    \min_{x\in\R^N} \norm{x}_C \quad\text{s.t. } \inner{A_i}{x} = b_i, i=1,...,m,
\end{equation}
where each $A_i$ represents a linear measurement.
The success of recovering a point $x$ is equivalent to the null space condition (\cite{Chandrasekaran_Recht_Parrilo_Willsky_2012})
\begin{equation}\label{null-space-condition}
    \ker(A)\cap D(C, [x]) = \set{0}.
\end{equation}
Then establishing a lower bound of $m$ amounts to measuring the size of $D(C, [x])$. 
In fact, a larger descent cone requires a smaller null space, hence more measurements.
The next proposition connects the threshold $m_0$ with the Gaussian width.
\begin{proposition}[\cite{Tropp_2015}]\label{proposition: Gaussian-width-number}
    Let all notations be as above, let $x_0\in \R^N\setminus\set{0}$,
    $A\in\R^{m\times N}$ is randomly drawn from i.i.d Gaussian entries and $b = Ax_0$. 
    Then if 
    \begin{equation}
        m\geq w(D(C, x_0))^2 + cw(D(C, x_0)).
    \end{equation}
    then with high probability, program (\ref{recovery-problem-norm}) has a unique solution $x_0$.
    $c>0$ is independent of $N$.
\end{proposition}

We can compare the lower bound of the sufficient number $m_0$
by using $\Thetabody{1}(I_2)$ and $\Thetabody{1}(I_\infty) = \thetabody{1}(I_\infty)$.

\begin{theorem}\label{theorem: compare 2-infty}
    Let $x$ be a rank-1 signed tensor, i.e. $x\in\Vcal_\R(I_\infty)$. 
    With notations in remark \ref{remark: descent-cone}, 
    \[
    D(\Thetabody{1}(I_2), [x])\supset D(\thetabody{1}(I_\infty),[x])\supset D(B_\infty, [x]).   
    \]
    In particular, for $\epsilon > 0$,
    the following statements have the implication relation $1.\implies 2.\implies 3.$.
    \begin{enumerate}
        \item $m$ measurements are sufficient to recover $x$ by $\norm{\cdot}_{\Thetabody{1}(I_2)}$ with probability at least $1-\epsilon$.
        \item $m$ measurements are sufficient to recover $x$ by $\norm{\cdot}_{\thetabody{1}(I_\infty)}$ with probability at least $1-\epsilon$.
        \item $m$ measurements are sufficient to recover $x$ by $\norm{\cdot}_{B_\infty}$ with probability at least $1-\epsilon$.
    \end{enumerate}
\end{theorem}
\begin{proof}
    It is straightforward that the Gaussian width is inclusive-preserving by definition,
    i.e. if $K_1\subset K_2$ as two convex cones, then $w(K_1)\leq w(K_2)$.
    It suffices to prove the inclusion relation in the theorem by proposition \ref{proposition: Gaussian-width-number}.

    For any ideal $I$,
    recall that $\Thetabody{1}(I)$ is the intersection of $0$-level sets($\set{x: f(x)\leq 0}$) of convex quadrics in the ideal $I$
    (Theorem \ref{Theorem: theta1-convex-quadrics}). 
    The normal cone can be expressed by these convex quadrics.
    \begin{equation}
        N(\Thetabody{1}(I), [x]) = \cone\set{\nabla f(x): f\in I, f\text{ convex quadrics}}.
    \end{equation}

    Now compare $I = I_2, I_\infty$.
    Returning to our settings of tensors,
    let $\nbf = (n_1,...,n_d)\in \N^d$ and $\R^{\nbf} = \R^{n_1\odots{\times}n_d}$.
    Recall their Gröbner basis from Section \ref{section: Gröbner},
    \[\Gcal_2 = \set{\sum_a x_a^2-1, x_ax_b-x_{a\wedge b}x_{a\vee b}, a,b\in[\nbf]},
    \Gcal_\infty = \set{x_a^2-1, x_ax_b-x_{a\bwedge b}x_{a\bvee b}, a,b\in[\nbf]}.\]
    Simple degree argument indicates that, 
    the convex quadrics in the ideal can be expressed as a linear combination of polynomials in the Gröbner basis;
    hence they contain no linear terms.
    Taking the gradient operator $\nabla$ gives
    $\nabla(\sum_a x^2_a-1) = \nabla(\sum_a (x_a^2-1))$.
    Now since any binomials in $\Gcal_2$ are also included by $\Gcal_\infty$, 
    we have
    \[\Span\set{\Gcal_2} \subset \Span\set{\Gcal_\infty}.\]
    It follows that $N(\Thetabody{1}(I_2), [x])\subset N(\Thetabody{1}(I_\infty), [x]).$
    By equation \ref{equation: T = -N*} and
    the fact that dual operator is inclusion-reversing, we have
     \[
    D(\Thetabody{1}(I_2), [x])\supset D(\thetabody{1}(I_\infty),[x])\supset D(B_\infty, [x]).   
    \]
    The last inequality follows from the fact $B_\infty\subset\Thetabody{1}(I_\infty)$
    and $x\in\Vcal(I_\infty)$ is an extreme point for both bodies. 
\end{proof}

\begin{remark}\label{remark: compare 2-infty} 
    According to \cite[Corollary 3.14]{Chandrasekaran_Recht_Parrilo_Willsky_2012},
    with $m\geq O(n)$ measurements,
    it suffices to recover a rank-1 signed tensor $x$ by using $\norm{\cdot}_{B_\infty}$.
    Our numerical results below suggest $m\geq O(n)$ should also be sufficient for $\Thetabody{1}(I_2)$-norm.
    Then we may conclude $O(n)$ is optimal.
\end{remark}

\subsection{Computation of Gaussian width}\label{section: computation-Gaussian-width}

We compute the Gaussian width of $D(\Thetabody{1}(I_2), [x])$
for rank-1 tensor $x\in\Vcal(I_2)$. 
The arguments of theorem \ref{theorem: compare 2-infty} and the inclusion

\[D(\thetabody{1}(I_2),[x])\subset D(\Thetabody{1}(I_2),[x]) \]
make our concentration on $\Thetabody{1}(I_2)$ reasonable.
Briefly speaking,
if the number of measurements are sufficient to recover a rank-1 tensor $x\in \Vcal(I_2)$ ($x\in\Vcal(I_\infty)$) with $\Thetabody{1}(I_2)$,
then it is also sufficient to use $\thetabody{1}(I_2)$ ($\thetabody{1}(I_\infty)$).

The dual trick from convex geometry and Jensen's inequality
\cite{Chandrasekaran_Recht_Parrilo_Willsky_2012}\cite{Tropp_2015}
imply that
\[
w(D(C, x))^2 \leq \Exp{\dist(g, D(C, x)^*)}{g}^2 \leq \Exp{\dist(g, D(C, x)^*)^2}{g}.
\]
Note $D(C,x)^* = -N(C,x)$. Next, we describe the normal cones of rank-1 tensors. 
Examples of illustrating the computation for sparse vectors and low rank matrices 
can be found in \cite[Section 4]{Tropp_2015}. 
We focus on rank-1 tensors within $\Vcal(I_2)\subset \R^\nbf = \R^{n_1\odots{\times} n_d}$.
It is known that $\Ogroup(n_1)\otimes\cdots \otimes \Ogroup(n_d)$ acts on $\Thetabody{1}(I_2)$,
which allows us to rotate every rank-1 tensor to a canonical form, i.e. $x_0 = [1,0,\cdots,0]\in\R^\nbf$.
We abbreviate $N(\Thetabody{1}(I_2), x_0)$ as $N(x_0)$. 
From now on, we use $x$ as variables for polynomials in $\R[X]$.
Next lemma is then a translation of lemma \ref{lemma: normal-cone-equiv} for $N(x_0)$.

\begin{lemma}\label{lemma: Nx-I2}
    $u\in N(x_0)$ if and only if one of the following holds.
    \begin{enumerate}
        \item There exists convex quadrics $f\in I, r>0$ such that $u = r\nabla f(x_0)$.
        \item There exists $r>0$ such that $1+\inner{ru}{x}\in \Sigma_1(I)\subset \R[X]/I$
        and $ru_{1...1} = 1$.
    \end{enumerate}
\end{lemma}
To furthermore understand the structure of this cone $N(x_0)$,
we recall that in matrix case (\cite[Example 4.4]{Tropp_2015}),
the elements of the normal cone at $[1,0,...,0]$ always admit some zeros, i.e
\begin{equation}\label{equation: Nx-matrix}
    N(x_0) = \cone\set{\begin{pmatrix}
        1&0\\
        0&A
    \end{pmatrix}: \norm{A} \leq 1}.   
\end{equation}
In tensor case, we have the following lemma.
\begin{lemma}
    Let $\nbf = (n_1,...,n_d)\in\N^d$ and $\Ical_0 = \set{b\in[\nbf]: \#\set{i\in [d]:b_i\not = 1} = 1}$,
    then for any $u\in N(x_0)$, $b\in \Ical_0$, $u_b = 0$.
\end{lemma}
\begin{proof}
    With the Gröbner basis, we know that a convex quadric $f\in I_2$ is
a linear combination of polynomials in $\Gcal_2$. Then $f(x) = x^TAx+C$. 
For any $u\in N(x_0)$, by the second statement in lemma \ref{lemma: Nx-I2},
there exists a convex quadric $f\in I_2$ s.t $u = r\nabla f(x_0) = 2rAx_0$.
Since $x_0 = [1,0,...,0]$,
$u$ is a positive scalar of the first row of $A$.
Now let $b\in\Ical_0$ and $a = (1,...,1)\in[\nbf]$. 
That $A_{a,b}$ will always vanish follows from the fact that $x_a x_b$ will never be a term in any elements of $\Gcal_2$.
Therefore $u_b = 2r A_{a,b} = 0$. 
\end{proof}

Let $a = (1,...,1)\in[\nbf]$, the first index of all indexes. 
Define $\Ical := [\nbf]\setminus (\Ical_0\cup{a})$.
Let $g\in\R^\nbf$ be a vector. 
For any index subset $\Jcal\subset [\nbf]$,
$g_\Jcal$ denotes the projection of $g$ onto these indices.
Motivated by \cite[Example 4.4]{Tropp_2015},
we calculate in the following way,
\begin{equation*}
    \Exp{\dist(g, N(x_0))^2}{g} = 
    \Exp{\inf_{\tau>0, u\in N(x_0), u_a = 1}
    |g_a - \tau|^2 + \norm{g_{\Ical_0}}_{2}^2
    +\norm{g_\Ical-\tau u_\Ical}_{2}^2}{g}
\end{equation*}
Define $N_\Ical:= \set{u_\Ical\in \R^{|\Ical|}: u\in N(x_0),u_a = 1, u_{\Ical_0} = 0}$.
Then the above expression is actually 
\[
    \Exp{\dist(g, N(x_0))^2}{g} = 
    \Exp{\inf_{\tau>0} |g_a - \tau|^2 
    + \norm{g_{\Ical_0}}_{2}^2
    +\inf_{\tau>0, u_\Ical\in N_\Ical}\norm{g_\Ical-\tau u_\Ical}_{2}^2}{g}
\]
Take $\tau = \gamma_{N_\Ical}(g_\Ical)$, the gauge function of $N_\Ical$,
i.e. $\gamma_{N_\Ical}(g_\Ical) = \inf\set{r>0: g_\Ical\in rN_\Ical}$.
Under this setting, if $N_\Ical$ is full dimensional in $\R^\Ical$ 
and contains the origin in its interior (proposition \ref{ref: NI full-dim}),
the third term above can always vanish.
This leads us to an upper bound of the Gaussian width with the expectation of this gauge function,
i.e.
\begin{equation}\label{inequality: Gaussian-width-gamma-NJ}
    \Exp{\dist(g, N(x_0))^2}{g} \leq 
    \Exp{|\Ical_0|+1+\gamma_{N_\Ical}(g_\Ical)^2}{g}.
\end{equation}
To achieve this inequality, we prove the following proposition.
\begin{proposition}\label{ref: NI full-dim}
    $N_\Ical\subset \R^\Ical$ defined above is indeed full dimensional and contains
    the origin as an interior point.
\end{proposition}
\begin{proof}
    Let $a = (1,...,1)\in[\nbf]$ and $b\in\Ical\subset[\nbf]$.
    Let $e_a, e_b$ be the standard vector in $\R^\nbf$ respectively. 
    Then $v\in N_\Ical$ if and only if $v+e_a\in N(x_0)$ by briefly extend $v$ to $\R^\nbf$ with zeros.
    We now argue that $(\pm e_b)_\Ical \in N_\Ical$. Then the proposition follows.

    In fact, for any $b\in \Ical$, it is necessary that $\#\set{i\in [d]: b_i\not = 1} > 1$.
    Suppose $b_1, b_2\not = 1$. Then the matricization with respect to $\set{1},\set{2,...,d}$
    of $u_0 = e_a \pm e_b$ have only two $\pm 1$s. One is at the position $(1,1)$ the other 
    is neither at the first row nor the first column because $b_1, b_2\not = 1$.
    Then this matricization of $u_0$ is indeed of the form of eq.\ref{equation: Nx-matrix}.
    It means $u_0$ is in the normal cone of the nuclear norm ball for matrices of size $\R^{n_1\times\cdots \times n_d}$
    at the rank-1 matrix $[1,0,...,0]$. Now let $J_2$ be the defining ideal of nuclear norm in this matricization.
    $J_2\subset I_2$ since a tensor is of rank-1 if and only if every matricization has rank-1 (by lemma \ref{lem-rk-1-eq}).
    Since matrix nuclear 2-norm is theta-1 exact by \cite{Rauhut_Stojanac_2017}.
    Lemma \ref{lemma: Nx-I2} tells us there exists a convex quadrics in $J_2\subset I_2$ such that
    $u_0$ is a positive scalar of $\nabla f(x)$. Then $u_0\in N(x_0)$ which means $(\pm e_b)_\Ical \in N_\Ical$.
\end{proof}
Now we explain that the computation of $\gamma_{N_\Ical}$ is indeed a semidefinite program.
This allows us to numerically estimate its expectation in the next section.
In fact, $v\in N_\Ical$ if and only if 
$1+x_a +\inner{v}{x_\Ical}\in\Sigma_1(I)$.
Then $g_\Ical\in rN_\Ical$ if and only if
\begin{equation}
    r+rx_a +\inner{g_\Ical}{x_\Ical}\in\Sigma_1(I).
\end{equation}
We have already demonstrated in Section \ref{section: dual body and k-sos} that
determining membership of $\Sigma_1(I)$ is a semidefinite program.

Suppose $n_i = n$ for all $i$.
We calculate $|\Ical_0| = \sum_i (n_i-1) \sim O(dn)$. 
Numerical results suggest that
$\Exp{\gamma_{N_\Ical}(g_\Ical)}{g} \sim O(n)$ as well, hence we conjecture that $O(n)$ number of measurements are enough 
to recover rank-1 tensors with upper theta-1 nuclear 2-norm.
Consequently, $O(n)$ should also be sufficient for $\thetabody{1}(I_2)$ and $\thetabody{1}(I_\infty)$
to recover the extreme points on their unit ball.

\subsection{Numerical Results}\label{section: numerical-results}

To solve the semidefinite program, 
we use JuMP(\cite{JuMP}) in Julia with the solver 
SCS(\cite{scs}).

\subsubsection*{Comparison between Theta-1 Nuclear $p$-Norm on low-rank signed tensors}

Figure \ref{figure: compare-signed} visualizes the results of theorem \ref{theorem: compare 2-infty}.
In the theorem, 
we argued that to recover rank-1 signed tensors, i.e. points in $\Vcal(I_\infty)$,
$\thetabody{1}(I_\infty)$ norm requires less measurements than $\Thetabody{1}(I_2)$.
To numerically approve this,
we randomly generated 50 tensors of size $4\times 4\times 4$ of rank up to 3 through linear combinations of rank-1 signed tensors.
We say the recovery is successful if the relative error between the original tensor and the output tensor is smaller than $10^{-3}$.
Figure \ref{fig:compare} shows the number of measurements required for recovery with different norms. The $x$-axis stands for the different experiments on these 50 tensors
while the $y$-axis exhibits the number of measurements required for successful recovery.
The results clearly show that $\thetabody{1}(I_\infty)$ requires less measurements.

Additionally, figure \ref{fig: time-compare} compares the computing time for recovering rank-1 signed tensors.
We fix the tensor where all entries are $1$.
Figure \ref{fig: time-compare} presents the time used as the number of measurements
increases from 1 to 50.
First, both curves grow at the beginning 
since the measurements are insufficient;
however, once the measurements are sufficient,
the computing time drops significantly.
The figure also reveals a trade-off between the choice of these two norms.
Specifically, 
using $\thetabody{1}(I_\infty)$ requires less measurements 
while employing $\thetabody{1}(I_2)$ tends to consume less time.

\begin{figure}[htbp]
\centering

    \begin{minipage}[t]{.5\linewidth}
        \includesvg[scale = 0.4]{Compare/success_m.svg}
        \subcaption{Comparison of number of measurements required}
        \label{fig:compare}
    \end{minipage} 
    \begin{minipage}[t]{.45\linewidth}
        \includesvg[scale = 0.4]{Compare/time_compare.svg}
        \subcaption{Computing time comparison}
        \label{fig: time-compare}
    \end{minipage} 

\caption{Comparison of $\thetabody{1}(I_\infty)$ and $\thetabody{1}(I_2)$ on signed tensors}
\label{figure: compare-signed}
\end{figure}

Furthermore, we did similar experiments on general low-rank tensors which are not necessarily a linear combination of signed tensors. The results are shown in figure \ref{fig:compare on random}.
Note that $\Vcal(I_\infty)\subset \sqrt{n_1\cdots n_d}\Vcal(I_2)$ as a "0-measure" subset.
Clearly, $\thetabody{1}(I_\infty)$ does not work on low-rank tensors in general, since
it requires almost $m=n_1\cdots n_d$ measurements.

\begin{figure}
    \centering
    \includesvg[width=0.5\linewidth]{Compare/random_success_m.svg}
    \caption{Number of measurements required by $\thetabody{1}(I_\infty)$ and $\thetabody{1}(I_2)$ on general tensors}
    \label{fig:compare on random}
\end{figure}

\subsubsection*{Sufficient measurements estimation of Theta-1 Nuclear 2-Norm}

We present a numerical estimation of the lower bound of sufficient number of measurements $m_0$
with respect to the $\Thetabody{1}(I_2)$ norm. 
Instead of directly estimating $m_0$
by generating examples with increasing measurements until the probability of successful recovery reaches $95\%$, 
we computed the Gaussian width as an alternative due to computational expense and estimation inaccuracy.
For instance, iterating the number of measurements from 10 to 50 and repeating the recovery process 100 times would require solving 4,000 semidefinite programs. 
Furthermore, this computation must be repeated for tensors of size $n\times n\times n$ with 
n ranging from 2 to 9. 
Additionally, the moment matrix has a size of $(N+1)\times (N+1)$ with $N = n^3$, 
further causing huge computational complexity.
Recall inequality \ref{inequality: Gaussian-width-gamma-NJ}
\[
    w(D(\Thetabody{1}(I_2),[x_0])\leq \Exp{|\Ical_0|+1+\gamma_{N_\Ical}(g_\Ical)^2}{g}.
\]
As argued, $|\Ical_0|\sim O(nd)$ and computing $\gamma_{N_\Ical}(g_\Ical)^2$ is a semidefinite program.
To obtain an average value we compute $\gamma_{N_\Ical}(g_\Ical)^2$ for $d=3, n=1,...,9$ and $d=4, n=1,...,8$ for 100 times. 
The results shown in the figure \ref{figure: gamma_NJ}
reveal that the trends are quite linear.
This suggests that
\begin{equation}
    w(D(\Thetabody{1}(I_2),x_0))^2 \leq O(n).
\end{equation}
Consequently, it follows that $m\geq O(n)$ should be sufficient for recovering rank-1 tensors
---with either $\thetabody{1}(I_2)$ or nuclear norm itself--- or for recovering rank-1 signed tensors with $\thetabody{1}(I_\infty)$.
\begin{figure}[htbp]
\centering

    \begin{minipage}[t]{.5\linewidth}
        \includesvg[scale = 0.4]{gamma_K_rank1/rank1_order3.svg}
        \subcaption{$n\times n\times n$ tensors}
    \end{minipage} 
    \begin{minipage}[t]{.45\linewidth}
        \includesvg[scale = 0.4]{gamma_K_rank1/rank1_order4.svg}
        \subcaption{$n\times n\times n\times n$ tensors}
    \end{minipage} 

\caption{Estimation of $\Exp{\gamma_{N_\Ical}(g_\Ical)^2}{g}$ for tensors of different order = 3, 4}
\label{figure: gamma_NJ}
\end{figure}

\quad\\ \quad\\ \quad\\ \quad\\
\bibliographystyle{unsrt}
\bibliography{library}

\end{document}